\documentclass[10pt,reqno]{amsart}

\usepackage{amssymb,amsmath,amsthm,amsxtra,amsfonts}
\usepackage{graphicx}

\usepackage{setspace}

\usepackage{hyperref}

\usepackage{graphics}
\usepackage{graphicx}
\usepackage{amsmath,amssymb}
\usepackage{setspace}

\newtheorem{theorem}{Theorem}[section]

\newtheorem{lemma}[theorem]{Lemma}

\numberwithin{equation}{section}
\newtheorem{definition}[theorem]{Definition}

\DeclareMathOperator{\diag}{diag}
\theoremstyle{remark}
\newtheorem*{remarks}{Remarks}
\newtheorem*{remark}{Remark}

\newcommand{\bbC}{{\mathbb{C}}}
\newcommand{\bbD}{{\mathbb{D}}}

\newcommand{\bbR}{{\mathbb{R}}}

\newcommand{\bbZ}{{\mathbb{Z}}}

\newcommand{\fre}{{\mathfrak{e}}}

\newcommand{\calJ}{{\mathcal J}}

\newcommand{\calR}{{\mathcal R}}
\newcommand{\calS}{{\mathcal S}}
\newcommand{\calT}{{\mathcal T}}

\newcommand{\bdone}{{\boldsymbol{1}}}
\newcommand{\bdnot}{{\boldsymbol{0}}}

\newcommand{\lla}{\left\langle\!\left\langle}
\newcommand{\rra}{\right\rangle\!\right\rangle}

\newcommand{\tr}{\text{\rm{Tr}}}

\newcommand{\spann}{\text{\rm{span}}}

\newcommand{\ran}{\text{\rm{Ran}\,}}

\newcommand{\beq}{\begin{equation}}
\newcommand{\eeq}{\end{equation}}
\newcommand{\ba}{\begin{align*}}
\newcommand{\ea}{\end{align*}}
\newcommand{\veps}{\varepsilon}

\DeclareMathOperator{\imag}{Im}

\DeclareMathOperator*{\esssup}{ess\,supp}

\DeclareMathOperator*{\res}{Res}

\begin{document}
\title[Meromorphic continuations and periodic Jacobi matrices]{Meromorphic continuations of finite gap Herglotz functions and
periodic Jacobi matrices}
%\titlerunning{Meromorphic continuations and periodic Jacobi matrices}
\author{Rostyslav Kozhan %\and Second author\inst{2}% etc
}                     % Do not remove
%
%\institute{ University of California, Los Angeles; Department of Mathematics; Los Angeles, CA 90095, USA. \email{kozhan@math.ucla.edu}
%\and the second here
%}                       % Do not remove
%
\date{Submitted: Oct 17, 2012; revised: Sept 13, 2013 \\ Email: kozhan@math.ucla.edu}
\address{University of California, Los Angeles\\
Department of Mathematics\\
Los Angeles, CA 90095, USA}
% Add name of the expert who has communicated your paper
%\communicated{Michael Aizenman}
%
\maketitle
\begin{abstract}
%Let $\mu$ be a finite measure on the real line whose essential support $\fre$ is a finite number of
%intervals (with equal equilibrium measures). Then the associated Herglotz function $m(z)=\int (x-z)^{-1} d\mu(x)$ (also called the Borel transform of $\mu$)
%is meromorphic on $\mathbb{C} \setminus \fre$. $\mathbb{C} \setminus \fre$ can be viewed as
%one of the two sheets of a natural Riemann surface. A necessary and sufficient condition is found
%for the function $m$ to have a meromorphic extension to some region on the second sheet of this
%surface. The domain of meromorphicity can be described explicitly in terms of the distance of the
%associated Jacobi matrix from the isospectral torus of periodic Jacobi matrices corresponding
%to the set $\fre$.
%
%\medskip
%or
%\medskip

%We find a necessary and sufficient condition for a finite gap Herglotz function $m$ to be the Borel transform of the spectral measure of a Jacobi matrix with the prescribed ``distance'' from the isospectral torus $\calT_\fre$ of periodic Jacobi matrices associated with a given finite gap set $\fre$ (with all gaps open). The condition is in terms of meromorphic continuations of the function $m$ to a natural Riemann surface $\calS_\fre$ and the structure of  poles and zeros of $m$. The analogous result is also established for the Borel transform of the spectral measure of eventually periodic Jacobi matrices.

We find a necessary and sufficient condition for a Herglotz function $m$ to be the Borel transform of the spectral measure of an exponentially decaying perturbation of a periodic Jacobi matrix. The condition is in terms of meromorphic continuation of $m$ to a natural Riemann surface and the structure of its zeros and poles.

The analogous result is also established for the Borel transform of the spectral measure of eventually periodic Jacobi matrices.

This paper generalizes the corresponding result from \cite{K_jost} for exponentially decaying perturbations of the free Jacobi matrix.
\end{abstract}

%%%%%%%%%%%%%%%%%%%%%%%%%%%%%%%%%%%%%%%%%%%%%%%%%%%%%%%%%%%%%%%%%%%%%%%%%%%%%%%%%%%%%%%%%

%\begin{keyword}  Meromorphic continuations \sep periodic orthogonal polynomials



%\end{keyword}

%\end{frontmatter}

%%%%%%%%%%%%%%%%%%%%%%%%%%%%%%%%%%%%%%%%%%%%%%%%%%%%%%%%%%%%%%%%%%%%%%%%%%%%%%%%%%%%%%%%%%%%%%%%%%

\begin{section}{Introduction}\label{Intro}

Let $\mu$ be a probability measure on the real line $\bbR$ with compact support. Denote by
\begin{equation}\label{intro1}
m(z)=\int \frac{d\mu(x)}{x-z}, \quad z\notin\esssup\mu
\end{equation}
the Borel transform (also sometimes referred as the Stieltjes transform) of $\mu$.
It is a Herglotz function: if $\imag z>0$ then $\imag m(z)>0$.

Assuming $\mu$ is a non-trivial measure, i.e., not supported on finitely many points, we can apply the Gram--Schmidt algorithm to orthonormalize the sequence of polynomials $\{x^n\}_{n=0}^\infty$. Let the resulting sequence of orthonormal polynomials be $\{p_n(x)\}_{n=0}^\infty$. They satisfy the Szeg\H{o} recurrence
\begin{equation}\label{intro2}
x p_n(x)=p_{n+1}(x)a_{n+1} +p_n(x)b_{n+1}+p_{n-1}(x) a_n, \quad n=1,2,\ldots,
\end{equation}
for some sequences of real numbers $a_n>0$ and $b_n\in\bbR$, called the Jacobi coefficients. In fact, if we put $p_{-1}(x)\equiv 0$, then \eqref{intro2} holds for $n=0$ too. Now one can see that the operator of multiplication by $x$ in $L^2(\mu)$ in the basis $\{p_n(x)\}_{n=0}^\infty$ has the form
\begin{equation}\label{intro3}
\calJ=\left(
\begin{array}{cccc}
b_1&a_1&{0}&\\
a_1 &b_2&a_2&\ddots\\
{0}&a_2 &b_3&\ddots\\
&\ddots&\ddots&\ddots\end{array}\right).
\end{equation}
This three-diagonal matrix is called the Jacobi matrix associated with the measure $\mu$. One can recover $\mu$ from $\calJ$ by finding the spectral measure of $\calJ$ corresponding to the vector $\delta_1=(1,0,0,\ldots)^T$. Thus from the operator viewpoint, function~\eqref{intro1} is just the $(1,1)$-entry of the resolvent of $\calJ$, also sometimes referred to as a Green's function in spectral theory.

The theme of this paper is that certain analytic properties of $m$ determine (in an if and only if fashion) how close $\calJ$ is to being periodic. The prototype for this is the following result from \cite{K_jost}.

The simplest Jacobi matrix is the one with constant Jacobi coefficients. After translating and scaling we may consider $a_n=1$, $b_n=0$, $n\ge1$. We will refer to this matrix as the free Jacobi matrix. The Borel transform of $\mu$ corresponding to the free Jacobi matrix is
\begin{equation}\label{m_free}
m(z)=\frac{-z+\sqrt{z^2-4}}{2},
\end{equation}
with the principal branch for the square root. %This expression easily follows from the recursion relation \eqref{m_recur} and the fact that $m(z)$  must be $O(z^{-1})$ at infinity.

%Note that the function \eqref{m_free} has a meromorphic continuation to a natural Riemann surface: take two copies $\calS_+$, $\calS_-$ of $\bbC\cup\{\infty\}\setminus [-2,2]$ and glue them together so that the top part of the slit $[-2+i0,2+i0]$ on the first sheet $\calS_+$ is glued to the bottom part of the slit $[-2-i0,2-i0]$ on the second sheet $\calS_-$, and vice versa.

Note that the function \eqref{m_free} has a meromorphic continuation to the hyperelliptic Riemann surface associated with the polynomial $z^2-4$. Informally one may think of this surface as two sheets of $\bbC\cup\{\infty\}\setminus [-2,2]$ glued together along the slit (see Section \ref{Prelim_Finite} for more details). It was shown in \cite[Thm 3.8]{K_jost}, that if a Jacobi matrix is ``exponentially close'' to being free (in the sense of \eqref{expo}), then its Borel transform $m$ has a meromorphic continuation through $[-2,2]$ to an explicit region on the second sheet. Instead of continuations of $m$ through $[-2,2]$, we can equivalently consider continuations of
\begin{equation}\label{capM}
M(z)=m(z+z^{-1}), \quad z\in\bbD=\{z:|z|<1\}
\end{equation}
through $\partial\bbD=\{z: |z|=1\}$. The result of \cite[Thm 3.8]{K_jost} then says that

%Let us set some notation. Denote by $\pi:\calS\to \bbC\cup\{\infty\}$ the ``projection map'' which extends the natural inclusions $\calS_+\hookrightarrow \bbC\cup\{\infty\}$, $\calS_-\hookrightarrow \bbC\cup\{\infty\}$. For every $z\in\bbC\cup\{\infty\}\setminus \{-2,2\}$ let $z_+$ and $z_-$ be the two preimages $\pi^{-1}(z)$ in $\calS_+$ and $\calS_-$ respectively (for $z\in\{-2,2\}$, $z_+$ and $z_-$ coincide). Let us write $z^\sharp$ to denote $\overline{z}_-$ if $z\in \calS_+$ and $\overline{z}_+$ if $z\in \calS_-$. Finally let  $m^\sharp(z)=\overline{m(z^\sharp)}$.

\begin{equation}\label{expo}
\limsup_{n\to\infty}\left(|b_n|+|1-a_n^2|\right)^{1/2n}\le R^{-1}
\end{equation}
if and only if \eqref{capM} satisfies
\begin{list}{}
{
\setlength\labelwidth{2em}%
\setlength\labelsep{0.5em}%
\setlength\leftmargin{2.5em}%
}
\item[(a)] $M$ has a meromorphic continuation to $\{z:|z|<R\}$;
\item[(b)] $M$ has no poles on $\partial\bbD$, except possibly at $\pm1$, where they are at most simple;
\item[(c)] $M(z)-M^\sharp(z)$ has no zeros in $\{z: R^{-1}<|z|<R \}$, except possibly at $\pm1$, where they are at most simple;
\item[(d)] if $M$ has a pole $z\in\bbD$ with $R^{-1}<|z|<1$, then $\bar{z}^{-1}$ is not a pole of $M$.
\end{list}
Here $M^\sharp(z)=\overline{M(\bar{z}^{-1})}$.
In fact, \cite{K_jost} also covers the case of matrix-valued measures. See  Lemmas \ref{p4_lm_M} and \ref{p4_lm_M_finite} below for the exact statement of the results.

%\begin{itemize}
%\item $m$ has a meromorphic continuation from $\calS_+$ to the set $E_R$ of the second sheet $\calS_-$, where $E_R$ is the interior of the preimage of $\{|z|<R\}$ under $z+z^{-1}$;
%\item $m$ has no poles in $\pi^{-1}(-2,2)$, and at most simple poles at $\pi^{-1}(\pm2)$;
%\item $m-m^\sharp$ has no zeros in $\pi^{-1}(E_R)$, except at $\pi^{-1}(\pm2)$, where they are at most simple;
%\item if $m$ has a pole $z\in\pi^{-1}(E_R\setminus[-2,2])$, then $z^\sharp$ is not a pole of $m$.
%\end{itemize}
%In fact, \cite{K_jost} also covers the case of matrix-valued measures. See below Lemma \ref{} for the exact statement.

%So, at least heuristically, Jacobi parameters behave in a similar fashion to the Taylor/Fourier coefficients. A reader interested in this analogy should also look at the recent Breuer--Simon results \cite{Breuer_Simon}.

The purpose of the current paper is to establish the analogue of the above equivalence for  perturbations of the periodic Jacobi matrices. Another way to put it, instead of considering $\esssup \mu = [-2,2]$  in this equivalence, we are extending it to the case $\esssup\mu = \cup_{j=1}^p [\alpha_j,\beta_j]$, a finite gap set.

One may put this result on its head and say that we obtain a criterion for a finite gap Herglotz function to have a meromorphic continuation without degeneracies of types (b), (c), (d).

The main results of this paper are stated in Theorem~\ref{p4_th_merom} and~\ref{p4_th_finite_merom} in Section~\ref{Results} below. One has to be careful in the periodic setting, since there is a whole multidimensional set of periodic Jacobi matrices that have the same spectrum. Theorem~\ref{p4_th_merom} corresponds to exponentially decaying perturbations of periodic Jacobi matrices, and Theorem~\ref{p4_th_finite_merom} is the refinement for the eventually periodic Jacobi matrices.
%Thus the analogue of \eqref{expo} is that a distance from a Jacobi matrix to the  \textit{set} of isospectral periodic Jacobi matrices is exponentially small. We discuss isospectral periodic Jacobi matrices and define the right notion of distance in Section ?.?.

The idea of the proof is to use the ``Magic Formula'' of Damanik--Killip--Simon (see Lemma~\ref{magic} in Appendix) which establishes a connection to the matrix-valued problem, and then apply the author's matrix-valued result (Lemmas~\ref{p4_lm_M}, \ref{p4_lm_M_finite}).

The present paper covers only the case when all the intervals $[\alpha_j,\beta_j]$ have equal equilibrium measure (the so-called ``all gaps open'' case). Even though this is a generic situation for the periodic Jacobi matrices, it would still be interesting to prove Theorems \ref{p4_th_merom} and \ref{p4_th_finite_merom} in the cases when measures of $[\alpha_j,\beta_j]$ are rational but unequal (``some gaps closed'' periodic setting), as well as when measures of $[\alpha_j,\beta_j]$ are not all rational (almost periodic setting). There is little doubt that similar theorems should still hold in these situations. However one would have to come up with a different approach to prove them, without the reliance on the Damanik--Killip--Simon formula.

For the background discussion of Jacobi matrices and orthogonal polynomials, see, e.g.,~\cite{Rice}. A textbook exposition of the theory of \textit{periodic} Jacobi matrices can be found there as well (in Chapter 5), along with an extensive historical discussion. Papers related to exponentially decaying perturbations of Jacobi matrices include (but are likely not limited to)~\cite{DS2,Geronimo_matrix,Geronimo,K_jost,S_merom_jost}.

The results of the present paper were completed and presented in 2010 (see the author's PhD thesis~\cite{K_thesis}). Later there appeared an independent series of papers by Iantchenko--Korotyaev~\cite{IK3,IK1,IK2}, who study eventually periodic Jacobi matrices, but from another perspective and using an entirely different approach. Their results are related to our Theorem~\ref{p4_th_finite_merom}. It should be noted however that the models and the results are different: Iantchenko--Korotyaev fix a periodic Jacobi matrix, which is assumed to be known, and then consider compact perturbations of it. In our approach, we fix the support of the spectrum and consider compact perturbations of any Jacobi matrix from the isospectral torus, without any other knowledge about it. %It does not seem obvious if their results can be deduced from our Theorem~\ref{p4_th_finite_merom} or conversely.

%The results of the present paper were completed and presented in 2010 (see the author's PhD thesis \cite{K_thesis}). Later there appeared an independent series of papers by Iantchenko--Korotyaev \cite{IK3,IK1,IK2}, who study eventually periodic Jacobi matrices, but from another perspective and using an entirely different approach. Their results are related to our Theorem~\ref{p4_th_finite_merom}, but not in a trivial manner.

The organization of the paper is as follows. Section~\ref{Prelim} contains the necessary definitions and preliminary information. Section~\ref{Results} contains the two main theorems. Section~\ref{Proofs} contains the proofs. Appendix contains all the necessary results from the theory of scalar and matrix-valued orthogonal polynomials (Appendix~\ref{Appendix_Orthogonal}), periodic Jacobi matrices and the connection between periodic and matrix-valued settings (Appendix~\ref{Appendix_Periodic}), general facts about matrix-valued functions (Appendix~\ref{Appendix_Matrixvalued}) and Herglotz functions (Appendix~\ref{Appendix_Herglotz}). A reader not familiar with the theory of orthogonal polynomials should familiarize (him/her)self with Appendices~\ref{Appendix_Orthogonal} and~\ref{Appendix_Periodic} prior to reading the proofs in Section~\ref{Proofs}. Otherwise, appendices can be used when referred to.

\end{section}

%%%%%%%%%%%%%%%%%%%%%%%%%%%%%%%%%%%%%%%%%%%%%%%%%%%%%%%%%%%%%%%%%%%%%%%%%%%%%%%%%%%%%%%%%%%%%%%%%
%%%%%%%%%%%%%%%%%%%%%%%%%%%%%%%%%%%%%%%%%%%%%%%%%%%%%%%%%%%%%%%%%%%%%%%%%%%%%%%%%%%%%%%%%%%%%%%%%

%%%%%%%%%%%%%%%%%%%%%%%%%%%%%%%%%%%%%%%%%%%%%%%%%%%%%%%%%%%%%%%%%%%%%%%%%%%%%%%%%%%%%%%%%%%%%%%%%%%%%%%%%%%%%%%%%%%%%%%%
\begin{section}{Preliminaries}\label{Prelim}

\begin{subsection}{Finite Gap Sets and Surface $\calS_\fre$}\label{Prelim_Finite}

In this subsection let us assume that $\mu$ is a probability measure, and its essential support  is a finite union of closed intervals (``{finite gap set}'')
\begin{equation}\label{p4_fre}
\esssup\mu=\fre=\bigcup_{j=1}^{g+1} [\alpha_j,\beta_j], \quad \alpha_j<\beta_j <\alpha_{j+1}.
\end{equation}
We will be referring to the collections of intervals $[\alpha_j,\beta_j]$ ($1\le j\le g+1$) as ``{bands}'', and $[\beta_j,\alpha_{j+1}]$ ($1\le j\le g$) as ``{gaps}''.
As we will see soon, the spectral measures of periodic Jacobi matrices have exactly this form.

Then $m$, defined by \eqref{intro1}, is a meromorphic function on $\bbC\setminus \fre$, and it is natural to ask if $m$ has a meromorphic continuation through $\fre$. Indeed, this is the analogue of the meromorphic continuation for the  $\fre=[-2,2]$ case that we discussed in the Introduction. Let us introduce the natural  Riemann surface that arises here.

\begin{definition}\label{s}
Assume $\fre$ is a finite gap set \eqref{p4_fre}. Define $\calS_\fre$ to be the hyperelliptic Riemann surface corresponding to the polynomial $\prod_{j=1}^{g+1}(z-\alpha_j)(z-\beta_j)$.
\end{definition}

We will not give the formal definition, which can be found in many textbooks (see, e.g., \cite[Sect 5.12]{Rice}). Informally $\calS_\fre$ can be described as follows.

Let $\bbC_+ = \{z: \imag z>0\}$, $\bbC_- = \{z:\imag z<0\}$. Denote $\calS_+$ and $\calS_-$ to be two copies of $\bbC\cup\{\infty\}$ with a slit along $\fre$ (include $\fre$ as a top edge and exclude it from the lower), and let $\calS_\fre$ be $\calS_+$ and $\calS_-$ glued together along $\fre$ in the following way: passing from $\bbC_+ \cap \calS_+$ through $\fre$ takes us to $\bbC_- \cap \calS_-$, and from $\bbC_- \cap \calS_+$ to $\bbC_+ \cap \calS_-$. It is clear that topologically $\calS_\fre$ is  an orientable manifold of genus $g$.

Let $\pi:\calS_\fre\to \bbC\cup\{\infty\}$ be the ``projection map'' which extends the natural inclusions $\calS_+\hookrightarrow \bbC\cup\{\infty\}$, $\calS_-\hookrightarrow \bbC\cup\{\infty\}$.
%Finally let us denote by $z_+$ and $z_-$ the two preimages $\pi^{-1}(z)$ of $z\in\bbC\cup\{\infty\}$ (for $z\in\cup_{j=1}^{g+1}\{\alpha_j,\beta_j\}$, $z_+$ and $z_-$ coincide).

The following notation will be used frequently throughout the paper.
\begin{definition}\label{sharp}
\hspace*{\fill}
\begin{list}{}
{
\setlength\labelwidth{2em}%
\setlength\labelsep{0.5em}%
\setlength\leftmargin{2.5em}%
}
\item[$\bullet$] For $z\in\bbC\cup\{\infty\}$, denote by $z_+$ and $z_-$ the two preimages $\pi^{-1}(z)$ in $\calS_+$ and $\calS_-$ respectively $($for $z\in\cup_{j=1}^{g+1}\{\alpha_j,\beta_j\}$, $z_+$ and $z_-$ coincide$)$.
\item[$\bullet$] Let $z^\sharp$ be $\left(\overline{\pi(z)}\right)_-$ if  $z\in \calS_+\setminus\pi^{-1}(\fre)$, and %$\bar{z}_+$
    $\left(\overline{\pi(z)}\right)_+$ if $z\in \calS_-\setminus\pi^{-1}(\fre)$. In order to make this continuous, we make the convention  $z^\sharp=z$ for $z\in \pi^{-1}(\fre)$.
\item[$\bullet$] Let  $m^\sharp(z)=m(z^\sharp)^*$.
\end{list}
\end{definition}
Here bar means complex conjugation, and $^*$ means Hermitian conjugation (later on we will allow $m$ to be a matrix-valued function).

\end{subsection}
%%%%%%%%%%%%%%%%%%%%%%%%%%%%%%%%%%%%%%%%%%%%%%%%%%%%%%%%%%%%%%%%%%%%%%%%%%%%%%%%%%%%%%%%%%%%%%%%%%%%%%%%%%%%%%%%%%%%%%%%%%%%%%%%%%%%%%%%%%%%%%%%%%%%%%%%%

\begin{subsection}{Periodic Orthogonal Polynomials on the Real Line}\label{Prelim_Periodic}

For all the proofs of the facts in this subsection, we refer the reader to \cite{Rice} and references therein. Some of the basics of the theory of  orthogonal polynomials, along with the necessary lemmas, are also listed below in Appendices~\ref{Appendix_Orthogonal} and~\ref{Appendix_Periodic}.

A  Jacobi matrix $\calJ$, see \eqref{intro3}, is called periodic   if there exists an integer $p\ge1$ such that
\begin{equation}\label{p4_periodic}
a_{n+p}=a_n, \quad b_{n+p}=b_n \quad \mbox{for all }n.
\end{equation}
One can also talk about two-sided Jacobi matrices, which are  operators on $\ell^2(\bbZ)$ of the same tridiagonal form as \eqref{intro3}, where  sequences $\{a_n, b_n\}_{n\in\bbZ}$ are now extended to the whole $\bbZ$. The same definition of periodicity \eqref{p4_periodic} applies to a two-sided Jacobi matrix as well. We will commonly use $(a_n,b_n)_{n=1}^\infty$, $(a_n,b_n)_{n\in\bbZ}$ %(or sometimes $(a_n,b_n)_{n=1}^p$)
as a notation for one-sided and two-sided Jacobi matrices, respectively.

For a one- or two-sided $p$-periodic Jacobi matrix one can associate the polynomial of degree $p$ with real coefficients
\begin{equation}\label{p4_discr}
\Delta(z)=\tr \left( \prod_{j=p}^1 \frac{1}{a_{j}} \left( \begin{array}{cc} z-b_j&-1\\a_j^2&0\end{array}  \right)\right),
\end{equation}
which is called the \textit{discriminant} of $\calJ$. %Note that this is just the trace of the update matrix corresponding to the vector $(u_{n+1},a_n u_n)^T$, where $u_n$ is a solution to the recurrence
%\begin{equation*}
%z u_n=a_n u_{n+1} + b_n u_n +a_{n-1} u_{n-1}.
%\end{equation*}

The polynomial $\Delta$ has numerous useful properties, some of which we list in Lemma~\ref{discriminant}. The most important for us here is that it determines the spectrum of $\calJ$.
%
%
%It turns out that
%\begin{equation}
%\Delta(z)=p_p(z)-a_p q_{p-1}(z),
%\end{equation}
%where $p_j,q_j$ are orthogonal polynomials of the first and second kind for $(a_n,b_n)_{n=1}^\infty$. Also,
%\begin{equation}
%\Delta(z)=\frac{1}{a_1\ldots a_p} \prod_{j=1}^p (z-b_j)+O(z^{p-2})
%\end{equation}
%holds.

It turns out that the spectrum of a two-sided periodic Jacobi matrix is purely absolutely continuous of multiplicity two, and
\begin{equation*}
\sigma((a_n,b_n)_{n\in\bbZ})=\Delta^{-1}([-2,2]).
\end{equation*}
Essential spectrum of a one-sided periodic Jacobi matrix is purely absolutely continuous of multiplicity one and we still have
\begin{equation*}
\sigma_{ess}((a_n,b_n)_{n=1}^\infty)=\Delta^{-1}([-2,2]).
\end{equation*}

In fact, $\Delta^{-1}([-2,2])$ is a finite gap set
\begin{equation}\label{p4_e}
%\Delta^{-1}([-2,2])=
\Delta^{-1}([-2,2])= \bigcup_{j=1}^p [\alpha_j,\beta_j] \equiv \fre , \quad \alpha_j<\beta_j \le\alpha_{j+1},
\end{equation}
where these intervals are allowed to touch. If some two intervals do touch $\beta_j=\alpha_{j+1}$, then this gap $[\beta_j,\alpha_{j+1}]$ is said to be \textit{closed}, and otherwise it is \textit{open}. Let $g$ be the number of open gaps (in other words, $\fre$ consists precisely of $g+1$ \textit{disjoint} closed intervals), which is consistent with the notation in the previous section.

Unlike the two-sided Jacobi matrices, the one-sided ones may have some point spectrum: $\sigma((a_n,b_n)_{n=1}^\infty)\setminus \Delta^{-1}([-2,2])$ may consist of up to $g$ eigenvalues, at most one per each open gap.

It turns out that if there exists at least one periodic Jacobi matrix $\calJ$ with $\sigma_{ess}(\calJ)=\fre$, then there exists a whole set of periodic Jacobi matrices satisfying the same property. In fact, this set is homeomorphic to $(S^1)^g$, a $g$-dimensional torus. This motivates the following definition.

\begin{definition}\label{torus}
The {isospectral torus} $\calT_\fre$  of  $\fre$ is the set of periodic Jacobi matrices $\calJ$ with $\sigma_{ess}(\calJ)=\fre$.
\end{definition}

We will view $\calT_\fre$ as a set of one-sided  $(a_n,b_n)_{n=1}^\infty$
 or two-sided  $(a_n,b_n)_{n\in\bbZ}$ matrices, depending on the context.

Denote $\rho_\fre$ to be the equilibrium (harmonic) measure of $\fre$.

There is an easy criterion for determining when a finite gap set $\fre$ is the (essential) spectrum of some periodic Jacobi matrix.

\begin{lemma} Let $\fre$ be a finite gap set \eqref{p4_e}.
\begin{list}{}
{
\setlength\labelwidth{2em}%
\setlength\labelsep{0.5em}%
\setlength\leftmargin{2.5em}%
}
\item[(a)] $\fre$ is the essential spectrum of some periodic Jacobi matrix if and only if the equilibrium measure of each of the $g+1$ disjoint intervals of $\fre$ is rational.
\item[(b)] $\fre$ is the essential spectrum of some $p$-periodic Jacobi matrix with all gaps open if and only if the equilibrium measures of each of the $p=g+1$ disjoint intervals of $\fre$ are equal $($and so equal to $1/p)$.
\end{list}
\end{lemma}
Note that (a) in the above lemma should be thought of as $p$ intervals of equal equilibrium measure, some of which may touch. So in a sense (which can be made rigorous), (b) in the generic subcase of (a).

As a side remark, if at least one of the $g+1$ disjoint intervals of $\fre$ has irrational equilibrium measure, then one can construct an \textit{almost periodic} Jacobi matrix with essential spectrum $\fre$. We will not be discussing them in this paper (see \cite[Chapt~9]{Rice} for more information).

Now let $\mu$ be the spectral measure of a periodic one-sided Jacobi matrix $\calJ=(a_n,b_n)_{n=1}^\infty$ with respect to the vector $\delta_1$, and let $m$ be its Borel transform \eqref{intro1}. Using the recursion-type relation \eqref{m_recur} and the periodicity of $\calJ$, one can easily obtain that $m$ satisfies a certain quadratic equation. In fact (see Lemma~\ref{discriminant}(ii)),
\begin{equation*}\label{p4_ee1.28}
m(z)=\frac{r(z)\pm \sqrt{\Delta^2(z)-4} }{t(z)}.
\end{equation*}
%where we choose the branch of the square root $\sqrt{\Delta^2(z)-4}=\Delta(z)+O(1/\Delta(z))$.
Here $r(z),t(z)$ are some polynomials in $z$. Comparing this with \eqref{p4_e}, one now sees that $m$ has a meromorphic continuation to the full  surface $\calS_\fre$, the genus $g$ hyperelliptic surface constructed in Definition \ref{s}.

Our aim is to show that spectral measures of exponentially decaying perturbations of periodic Jacobi matrices have Borel transforms $m$ that can be meromorphically continued from $\calS_+$ to a portion of $\calS_-$. In fact, up to some poles/zeros constraints, these are the only measures that have this property.

\end{subsection}
\end{section}
%%%%%%%%%%%%%%%%%%%%%%%%%%%%%%%%%%%%%%%%%%%%%%%%%%%%%%%%%%%%%%%%%%%%%%%%%%%%%%%%%%%%
%%%%%%%%%%%%%%%%%%%%%%%%%%%%%%%%%%%%%%%%%%%%%%%%%%%%%%%%%%%%%%%%%%%%%%%%%%%%%%%%%%%%%

\begin{section}{Results}\label{Results}

Let $\mathfrak{e}=\cup_{j=1}^p [\alpha_j,\beta_j]$, $\alpha_j<\beta_j<\alpha_{j+1}$, be such that each $[\alpha_j,\beta_j]$ has equal equilibrium measure (``open gaps case'').

Assume $\esssup\mu=\mathfrak{e}$, and let $m(z)=\int_\bbR \frac{d\mu(x)}{x-z}$.

Denote $\Delta$ to be the unique polynomial of degree $p$ such that $\mathfrak{e}=\Delta^{-1}[-2,2]$ (its existence follows from the discussion in Section~\ref{Prelim_Periodic}). Let $x(z)=z+z^{-1}$. For each $R>1$, let
\begin{equation*}\label{E_R}
\calS_R=\calS_+\cup \pi^{-1}(E_R),
\end{equation*}
where $E_R$ is the union of the interiors of the bounded components of the set $\Delta^{-1}(x(R\,\partial\bbD))$.

\begin{theorem}\label{p4_th_merom}
%Let $\mathfrak{e}=\cup_{j=1}^p [\alpha_j,\beta_j]$, $\alpha_j<\beta_j<\alpha_{j+1}$, is such that each $[\alpha_j,\beta_j]$ has equal equilibrium measure (``open gaps case'').

%Assume $\esssup\mu=\mathfrak{e}$, and let $m(z)=\int_\bbR \frac{d\mu(x)}{x-z}$.
Let $R>1$. The following are equivalent:
\begin{list}{}
{
\setlength\labelwidth{2em}%
\setlength\labelsep{0.5em}%
\setlength\leftmargin{2.5em}%
}
%\item[(i)] The associated  to $\mu$ Jacobi matrix $\calJ$ satisfies
%  \begin{equation*}
%  \limsup_{n\to\infty} (d_n(\calJ,\calT_{\mathfrak{e}}))^{1/2n} \le R^{-1},
%  \end{equation*}
%  where $\calT_\fre$ is the isospectral torus corresponding to $\fre$.
\item[(i)] The Jacobi matrix $(a_n,b_n)_{n=1}^\infty$ associated  with $\mu$  satisfies
  \begin{equation*}
  \limsup_{n\to\infty} \left(|a_n-a_n^{(0)}| + |b_n-b_n^{(0)}|\right)^{1/2n}  \le R^{-1},
  \end{equation*}
  where $\big(a_n^{(0)},b_n^{(0)}\big)_{n=1}^\infty$ is a periodic Jacobi matrix from $\calT_\fre$.
  %with $\sigma_{ess}((a_n^{(0)},b_n^{(0)})_{n=1}^\infty)=\fre$ $($i.e, $(a_n^{(0)},b_n^{(0)})_{n=1}^\infty \in\calT_\fre)$.
\item[(ii)]
  \begin{itemize}
  \item[(a)] $m$ has a meromorphic continuation to $\calS_R$; %$\calS_-\cap E_R$,
  \item[(b)] $m$ has no poles on $\pi^{-1}(\fre)$, except at $\pi^{-1}(\cup_{j=1}^{p}\{\alpha_j,\beta_j\})$, where they are at most simple;
  %$\mu$ has no pure points in $\fre$,
  \item[(c)] $m(z)-m^\sharp(z)$ has no zeros in $\pi^{-1}(E_R)$, except at $\pi^{-1}(\cup_{j=1}^{p}\{\alpha_j,\beta_j\})$, where they are at most simple; %(in particular $m(z_+)=m(z_-)=\infty$ in not allowed).
  \item[(d)] If $m$ has a pole at $z$ for $z\in \pi^{-1}(E_R\setminus \fre)$ then $z^\sharp$ is not a pole of $m$.
  \end{itemize}
\end{list}
\end{theorem}

\begin{theorem}\label{p4_th_finite_merom}
%Let $\mathfrak{e}=\cup_{j=1}^p [\alpha_j,\beta_j]$, $\alpha_j<\beta_j<\alpha_{j+1}$, is such that each $[\alpha_j,\beta_j]$ has equal equilibrium measure (``open gaps case'').
%Assume $\esssup\mu=\mathfrak{e}$, and let $m(z)=\int_\bbR \frac{d\mu(x)}{x-z}$.
The following are equivalent:
\begin{list}{}
{
\setlength\labelwidth{2em}%
\setlength\labelsep{0.5em}%
\setlength\leftmargin{2.5em}%
}
\item[(i)] The Jacobi matrix $(a_n,b_n)_{n=1}^\infty$ associated  with $\mu$  is eventually periodic, i.e., satisfies
  \begin{equation*}
%  a_n=a_{n+p}, \quad b_n=b_{n+p} \quad \mbox{ for all large }n,
(a_n,b_n)_{n=N}^\infty\in\calT_\fre \quad \mbox{ for  large }N.
  \end{equation*}
%  such that $(a_n,b_n)_{n=N}^\infty\in\calT_\fre$ for large $N$.
\item[(ii)]
%\hspace*{\fill}
  \begin{itemize}
%  \setlength{\itemindent}{2em}
%  \item[(a)]\begin{itemize}
%  \item[(ii)]
%  \item[(pp)]
%  \end{itemize}
  \item[(a)] $m$ has a meromorphic continuation to $\calS_\fre$;
  \item[(b)] $m$ has no poles on $\pi^{-1}(\fre)$, except at $\pi^{-1}(\cup_{j=1}^{p}\{\alpha_j,\beta_j\})$, where they are at most simple;
  %$\mu$ has no pure points in $\fre$,
  \item[(c)] $m(z)-m^\sharp(z)$ has no zeros in $\calS_\fre\setminus\{\pm\infty\}$, except at $\pi^{-1}(\cup_{j=1}^{p}\{\alpha_j,\beta_j\})$, where they are at most simple;
  \item[(d)] If $m$ has a pole at $z$ for $z\in \pi^{-1}(\bbC\setminus \fre)$ then $z^\sharp$ is not a pole of $m$.
  \end{itemize}
%$$
%\limsup_{n\to\infty} (d_n((a,b),\calT_\mathfrak{e}))^{1/2n}\le R^{-1}.
%$$
\end{list}
\end{theorem}
\begin{remarks}
1. Theorems \ref{p4_th_merom}, \ref{p4_th_finite_merom} for $p=1$ and \cite[Thm 3.8, 3.9]{K_jost} (see Lemmas~\ref{p4_lm_M}, \ref{p4_lm_M_finite}) for $l=1$ are identical.

2. $R=\infty$ in Theorem \ref{p4_th_merom} is, in fact, allowed, and this case is \textit{not} the same as Theorem \ref{p4_th_finite_merom}.

3. Let us try to understand conditions (a) through (d) in terms of the properties of the measure $\mu$. Condition (b) just says that $\mu$ has no pure points on $\fre$ (see Lemma~\ref{p4_ges}). % From (c) it follows that $\imag m(z)\ne0$ for $z$ in the interior of $\fre$, and that at the edges $m(z)- m^\sharp(z)$ has at most first order  zeros. Comparing these facts with Lemma~\ref{p4_ges}, this means that the  density $f(x)=\frac{d\mu}{dx}$ of $\mu$ is non-vanishing on $\fre$ except at the edges where it might be square root vanishing (recall that local coordinates of $\calS_\fre$ at the edges of $\fre$ are given in terms of $\sqrt{z-z_0}$, not $z-z_0$).
By the discussion after Lemma \ref{p4_prop}, the conditions (a) and (c) imply that $\mu$ has no singular continuous part; the absolutely-continuous density $f(x)=\frac{d\mu}{dx}$ has a meromorphic continuation to $\pi^{-1}(E_R)$, where it is non-vanishing except possibly the first order zeros at the band edges (recall that local coordinates of $\calS_\fre$ at the edges of $\fre$ are given in terms of $\sqrt{z-z_0}$, not $z-z_0$). However it is not so simple to express the condition (d) in terms of the properties of $\mu$ alone, since it %that $m$ has no symmetric poles
is influenced by both the absolutely continuous and pure point parts of $\mu$.

%3. Instead of demanding (d) to hold for $z\in \pi^{-1}(\bbC\setminus \fre)$ one could demand it also for the points  $z\in \pi^{-1}(\fre\setminus \cup_{j=1}^{p}\{\alpha_j,\beta_j\})$, which, given the convention $z^\sharp=z$ ($z\in\pi^{-1}(\fre)$), would simply mean that $m$ has no pole at these points. %This is equivalent to saying that there is no point mass in the interior of $\fre$.
%(b) however also demands that the poles at the band edges are at most simple. %there is no point mass at the edges of $\fre$ ($\lim_{\veps\downarrow 0} \veps m(z_0+i\veps)=0$), which means that there can be a pole of order $1$ there, but not of order $2$.

4. Below is an example how $E_R$ evolves as $R$ grows (the picture was generated using Wolfram Mathematica 7.0). Using the results of \cite[Chapt~5]{Rice}, it is easy to see that $E_R$ are precisely the interiors of the level sets of the logarithmic potential of the equilibrium measure for $\fre$.

\medskip

%\medskip
%{\begin{center}
%\includegraphics[scale=0.59]{Fig2.png}
%\end{center}
%}

%temporarily commented out:
{
\begin{center}
\resizebox{0.8\textwidth}{!}{%
  \includegraphics{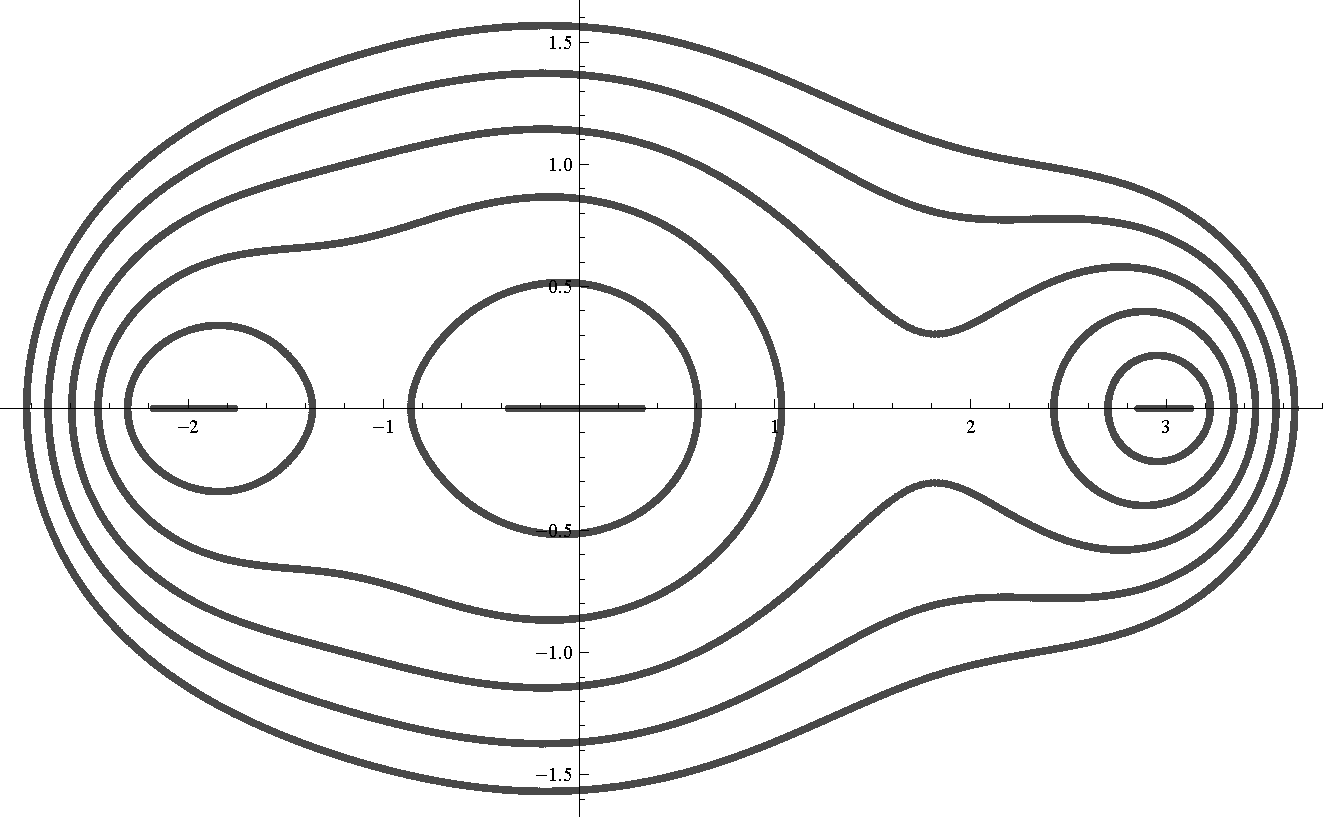}
}
\end{center}
}

\end{remarks}
\end{section}

%%%%%%%%%%%%%%%%%%%%%%%%%%%%%%%%%%%%%%%%%%%%%%%%%%%%%%%%%%%%%%%%%%%%%%%%%%%%%%%%%%%%%%%%%
%%%%%%%%%%%%%%%%%%%%%%%%%%%%%%%%%%%%%%%%%%%%%%%%%%%%%%%%%%%%%%%%%%%%%%%%%%%%%%%%%%%%%%%%%

\begin{section}{Proofs}\label{Proofs}

\begin{subsection}{Notation}\label{Proofs_Notation}

Let $\fre, \mu, m, \Delta, E_R$ be as in Section 3. Let $\calJ$ be the Jacobi matrix associated with $\mu$. As explained in Appendix~\ref{Appendix_Periodic}, $\Delta(\calJ)$ can be viewed as a block Jacobi matrix with $p\times p$ matrix entries.

Let $p_n(x)$, $q_n(x)$ be the orthonormal polynomials of the first and the second kind for $\calJ$ (see Appendix~\ref{Appendix_Orthogonal}), and $\mathfrak{p}_n(x)$, $\mathfrak{q}_n(x)$ be the right matrix-valued orthonormal polynomials of the first and the second kind for $\Delta(\calJ)$.

Denote by $\calS=\calS_\fre$ the (genus $p-1$) Riemann surface corresponding to $\fre$, and by $\calR=\calS_{[-2,2]}$ the (genus $0$) Riemann surface corresponding to $[-2,2]$ (i.e., the hyperelliptic surface corresponding to the polynomial $z^2-4$). We will denote both projections $\calS\to\bbC\cup\{\infty\}$ and $\calR\to\bbC\cup\{\infty\}$ by the same symbol $\pi$, in hopes that it should be unambiguous from the context.

Recall that $\calS_R=\calS_+\cup \pi^{-1}(E_R)$, where $E_R$ is the union of the interiors of the bounded components of $\Delta^{-1}(x(R\,\partial\bbD))$, where $x(z)=z+z^{-1}$. Denote $\calR_R=\calR_+\cup \pi^{-1}(F_R)$, where $F_R$ is the interior of the bounded component of $x(R\,\partial\bbD)$ (ellipse).

Let $\mu$ be the spectral measures for $\calJ$ with respect to $\delta_1=(1,0,0,\ldots)^T$. Let $\mu_\Delta$ be the $p\times p$ spectral measures of $\Delta(\calJ)$ with respect to $(\mathbf{1},\mathbf{0},\mathbf{0},\ldots)^T$.  Here $\mathbf{1}$, $\mathbf{0}$ are the $p\times p$ identity matrix and the $p\times p$ zero matrix, respectively.

Let $m$ be the Borel transform  of $\mu$. It is a meromorphic function on $\bbC\cup\{\infty\}\setminus \fre$. However we will view it as a meromorphic function on $\calS_+$ under the natural identification. Similarly let $\mathfrak{m}_\Delta$ be the $p\times p$ matrix-valued Borel transform  of $\mu_\Delta$. It is meromorphic on $\bbC\cup\{\infty\}\setminus [-2,2]$  by the spectral theorem. Indeed, $\Delta(\fre)=[-2,2]$. Again, we will view it as a meromorphic function on $\calR_+$.

As in Definition \ref{sharp}, let $z^\sharp$ be $\left(\overline{\pi(z)}\right)_-$ if $z\in \calS_+$ and $\left(\overline{\pi(z)}\right)_+$ if $z\in \calS_-$, with the convention  $z^\sharp=z$ for $z\in \pi^{-1}(\fre)$. Similarly, let $\lambda^\sharp$ be $\left(\overline{\pi(\lambda)}\right)_-$ if $\lambda\in \calR_+$ and $\left(\overline{\pi(z)}\right)_+$ if $\lambda\in \calR_-$, with the convention  $\lambda^\sharp=\lambda$ for $\lambda\in \pi^{-1}([-2,2])$. Let  $m^\sharp(z)=\overline{m(z^\sharp)}$ and  $\mathfrak{m}_\Delta^\sharp(\lambda)=\mathfrak{m}_\Delta(\lambda^\sharp){}^*$.

%Let  $m^\sharp(z_\pm)=\overline{m(\bar{z}_\mp)}$ and $m_\Delta^\sharp(\lambda_\pm)=m_\Delta(\bar{\lambda}_\mp)^*$. In order to make this continuous we make the convention that $\bar{z_-}=z_+$, $\bar{z_+}=z_-$ for $z\in\fre$. We will also sometimes write $z^\sharp$ to denote $\bar{z}_-$ if $z\in \calS_+$ and $\bar{z}_+$ if $z\in \calS_-$, and again, $z^\sharp=z$ for $\pi^{-1}(\fre)$ (and the same notation for $\lambda^\sharp$ with $\calS$ replaced by $\calR$).

Let $\{\gamma_j\}_{j=1}^{p-1}$ be the $p-1$ real solutions of $\Delta'(z)=0$ (they are indeed all real by Lemma \ref{discriminant}). Denote by %$\zeta_j=\Delta(\gamma_j)$, and
$\{\xi_j\}_{j=1}^N$  all of the preimages $\Delta^{-1}(\Delta(\gamma_j))$ (so the set $\{\xi_j\}_{j=1}^N$ contains all $\gamma_j$'s and finitely many of other points).

%It will be convenient to define $\calR_1\subset\calR$ as the union of $\calR_+\cap \pi^{-1}(\bbC_+)$, $\calR_-\cap \pi^{-1}(\bbC_-)$, and the interval $\pi^{-1}([-2,2])\cap\calR_+$ between them. Here $\bbC_+ =\{z: \imag z>0\}$, $\bbC_- =\{z: \imag z<0\}$. Similarly, let $\calR_2\subset\calR$ be the union of $\calR_-\cap \pi^{-1}(\bbC_+)$, $\calR_+\cap \pi^{-1}(\bbC_-)$, and the interval $\pi^{-1}([-2,2])\cap\calR_-$ between them. Clearly $\calR_1$, $\calR_2$ are simply-connected subsets of $\calR$ that have only $\pm2$ as common points.

Denote the $p$ inverse functions of $\Delta$ by $f_j$: $$\Delta(z)=\lambda \Rightarrow z=f_j(\lambda).$$ Initially we can define $f_j$ on $\bbC_+\cup \bbC_-\cup[-2,2]$ (the critical points of $\Delta$ are all in $(-\infty,-2)\cup(2,\infty)$), and then extend it to $(-\infty,-2)\cup (2,\infty)$ by demanding it to be continuous ``from above'', i.e., for $\lambda_0\in (-\infty,-2)\cup (2,\infty)$,
$$
f_j(\lambda_0)=\lim_{\bbC_+ \ni \lambda\to \lambda_0} {f}_j(\lambda).
$$
With this convention, we have that $f_j$ are functions defined everywhere on $\bbC$ with possible discontinuity only along $(-\infty,-2)\cup (2,\infty)$. Also note that for any $\lambda\in\bbC$ (including $(-\infty,-2)\cup (2,\infty)$), the set $\{f_j(\lambda): 1\le j\le p\}$ is equal to $\{z: \Delta(z)=\lambda\}$. In fact, if $\lambda\in\bbC\setminus\{ \Delta(\gamma_1),\ldots,\Delta(\gamma_{p-1}) \}$, then $f_j(\lambda)$ are all distinct for $j=1,\ldots, p$.

Counting zeros one can see that
\begin{equation}\label{disc_fact}
\Delta(z)-\lambda = c_0 \prod_{j=1}^p (z-f_j(\lambda))
\end{equation}
for some constant $c_0\in\bbR\setminus\{0\}$. In fact, $c_0 = \operatorname{Cap}(\fre)^{-p}$ (see \cite[Chapt 5]{Rice}), where $\operatorname{Cap}(\fre)$ stands for the logarithmic capacity of the set $\fre$. %We will not need this fact in the sequel.

Now let us ``lift'' the maps $f_j$. Define $\widetilde{f}_j$ to be the unique map $\calR \to \calS$  satisfying the conditions
\begin{align*}
\pi&\circ\widetilde{f}_j=f_j, \\
\widetilde{f}_j(\lambda) &\in \calS_+ \quad \mbox{if }\lambda\in \calR_+, \\
\widetilde{f}_j(\lambda) &\in \calS_- \quad \mbox{if }\lambda\in \calR_-.
\end{align*}
Note that each $\widetilde{f}_j$ is continuous everywhere except on $\pi^{-1}((-\infty,-2)\cup (2,\infty))$.

%Finally, extend $\widetilde{f}_j$ to $(-\infty,-2)\cup (2,\infty)$ on $\calR_+$ and $\calR_-$ by demanding it to be continuous ``from above'' (i.e., $\widetilde{f}_j(z_0)=\lim_{\calR_+ \cap \bbC_+ \ni z\to z_0} \widetilde{f}_j(z)$ for $z_0\in \calR_+\cap [(-\infty,-2)\cup (2,\infty)]$ and $\widetilde{f}_j(z_0)=\lim_{\calR_- \cap \bbC_+ \ni z\to z_0}\widetilde{f}_j(z)$ for $z_0\in \calR_-\cap [(-\infty,-2)\cup (2,\infty)]$). By doing this we are ensured that all $p$ of the preimages (counting multiplicities) $\Delta^{-1}(\lambda)$ are counted in by $\pi(\widetilde{f}_j(\lambda))$, $1\le j\le p$ for any $\lambda$.

Define $\widetilde{\Delta}:\calS\to\calR$ in the analogous way:
\begin{align*}
\pi&\circ\widetilde{\Delta}=\Delta, \\
\widetilde{\Delta}(z) &\in \calR_+ \quad \mbox{if }z\in \calS_+, \\
\widetilde{\Delta}(z) &\in \calR_- \quad \mbox{if }z\in \calS_-.
\end{align*}

Whenever we have any function  $g$ of complex variable, and $z\in\calS$, $\lambda\in\calR$, then we will occasionally write $g(z)$, $g(\lambda)$ instead of $g(\pi(z))$, $g(\pi(\lambda))$.

Throughout the paper, by a simple pole  of a matrix-valued meromorphic function $\mathfrak{m}(\lambda)$, we mean a point $\lambda_0$ where $\lim_{\lambda\to\lambda_0} (\lambda-\lambda_0)\mathfrak{m}(\lambda)$ exists and is a non-zero matrix.
%Similarly, $\lambda_0$ is a simple zero of a matrix-valued function $\mathfrak{m}(\lambda)$ if $\lambda_0$ is a simple pole of $\mathfrak{m}^{-1}$.
By a regular point of a function $\mathfrak{m}$, we mean a point $\lambda_0$ where $\lim_{\lambda\to\lambda_0}\mathfrak{m}(\lambda)$ exists.

\end{subsection}
%%%%%%%%%%%%%%%%%%%%%%%%%%%%%%%%%%%%%%%%%%%%%%%%%%%%%%%%%%%%%%%%%%%%%%%%%%%%%%%%%%%%%%%%%%%%%%%%%%5

\begin{subsection}{Lemmas}\label{Proofs_Lemmas}

\begin{lemma}\label{p4_may0}
For $\lambda\in \calR_+\setminus\pi^{-1}([-2,2])$,%$(\Delta(\calJ)-\lambda)^{-1} \prod_{j\ne l} (\calJ-\pi\circ f_j(\lambda))=(\calJ-f_l(\lambda))^{-1}$, in particular
\begin{equation}\label{p4_func}
m(\widetilde{f}_l(\lambda))=\left( c_0 (\Delta(\calJ)-\lambda)^{-1} \prod_{j\ne l} (\calJ-f_j(\lambda)) \, \delta_1,\delta_1
\right).
\end{equation}
\end{lemma}
\begin{proof}
Since $(x-f_l(\lambda))^{-1} = c_0 (\Delta(x)-\lambda)^{-1} \prod_{j\ne l} (x-f_j(\lambda))$, we obtain
\begin{equation}\label{p4_ee2.10}
(\calJ-{f}_l(\lambda))^{-1} = c_0(\Delta(\calJ)-\lambda)^{-1} \prod_{j\ne l} (\calJ-f_j(\lambda)) \quad \mbox{ for } \lambda\in\calR_+
\end{equation}
(note also that $\prod_{j\ne l} (\calJ-f_j(\lambda))$ is a finite-banded matrix, so the multiplication on the right-hand side is well-defined). Now using \eqref{p4_resolvent}, we obtain the result of the lemma. %The tilde on the left-hand side of \eqref{p4_func} is needed only because we are looking at $m$ as a function on $\calS$, rather than on $\bbC$.
\qed
\end{proof}

Note that \eqref{p4_func} allows one to continue $m$ using the continuation of $\mathfrak{m}_\Delta$, but not vice versa since we cannot invert the operator $\prod_{j\ne l} (\calJ-f_j(\lambda))$. There is a trick that will help us, though.

\begin{lemma}\label{p4_may1} For $\lambda\in \calR_+$,
\begin{multline}\label{p4_ee2.6}
\sum_{j=1}^p
% \left( \begin{array}{cccc} q_0(f_j(\lambda))+p_0(f_j(\lambda))m(f_j(\lambda))& q_1(f_j(\lambda))+p_1(f_j(\lambda))m(f_j(\lambda))&\cdots & q_{p-1}(f_j(\lambda))+p_{p-1}(f_j(\lambda))m(f_j(\lambda)) \\ q_1(f_j(\lambda))+p_1(f_j(\lambda))m(f_j(\lambda)) &q_1(f_j(\lambda))p_1(f_j(\lambda))+p^2_1(f_j(\lambda))m(f_j(\lambda)) &\cdots & q_{p-1}(f_j(\lambda))p_1(f_j(\lambda))+p_{p-1}(f_j(\lambda)) p_1(f_j(\lambda))m(f_j(\lambda)) \\ \vdots&\vdots&\ddots&\vdots \\ q_{p-1}(f_j(\lambda))+p_{p-1}(f_j(\lambda))m(f_j(\lambda)) & q_{p-1}(f_j(\lambda))p_1(f_j(\lambda))+p_{p-1}(f_j(\lambda)) p_1(f_j(\lambda))m(f_j(\lambda)) &\cdots & q_{p-1}(f_j(\lambda))p_{p-1}(f_j(\lambda))+p^2_{p-1}(f_j(\lambda))m(f_j(\lambda)) \end{array}
%\right)\\
 \left( \begin{array}{cccc} q_0+p_0m& q_1+p_1m&\cdots & q_{p-1}+p_{p-1}m \\ q_1+p_1m &q_1p_1+p^2_1m &\cdots & q_{p-1}p_1+p_{p-1} p_1m \\ \vdots&\vdots&\ddots&\vdots \\ q_{p-1}+p_{p-1}m & q_{p-1}p_1+p_{p-1} p_1m &\cdots & q_{p-1}p_{p-1}+p^2_{p-1}m \end{array}\right)(\widetilde{f}_j(\lambda)) \\
 =
 \mathfrak{m}_\Delta(\lambda)(S_{11}+\mathfrak{p}^R_1(\lambda)S_{21})+ A_1^{-1}{}^* S_{21},
\end{multline}
where $S_{ij}$ is the $(i,j)$-th $p\times p$ block entry of $\Delta'(\calJ)$, and $p_j, q_j$ are  the first and second kind polynomials for $\calJ$.
\end{lemma}
%\begin{remark}
%Until the end of the paper, whenever $z\in \calS$, then by $p_k(z)$ and $q_k(z)$ we mean $p_k(\pi(z))$ and $q_k(\pi(z))$ respectively. Similarly, if $\lambda\in\calR$, then by $\mathfrak{p}_k(\lambda)$ and $\mathfrak{q}_k(\lambda)$ we mean $\mathfrak{p}_k(\pi(\lambda))$ and $\mathfrak{q}_k(\pi(\lambda))$ respectively.
%\end{remark}
\begin{proof}
Sum the equalities \eqref{p4_ee2.10} from $l=1$ to $p$:
\begin{equation}\label{p4_ee2.12}
\begin{aligned}
\sum_{l=1}^p (\calJ-\widetilde{f}_l(\lambda))^{-1} &= c_0 (\Delta(\calJ)-\lambda)^{-1} \sum_{l=1}^p \prod_{j\ne l} (\calJ-f_j(\lambda)) \\
&= (\Delta(\calJ)-\lambda)^{-1} \Delta'(\calJ).
\end{aligned}
\end{equation}
The last equality comes from
\begin{equation}\label{p4_ee2.13}
c_0 \sum_{l=1}^p \prod_{j\ne l} (x-f_j(\lambda))= \Delta'(x)
\end{equation}
(to see this, just differentiate \eqref{disc_fact}).

Now using \eqref{p4_resolvent}, and taking the top-left $p\times p$ block of both sides of \eqref{p4_ee2.12}, we obtain
\begin{equation*}
\mbox{LHS of \eqref{p4_ee2.6}}= \mathfrak{m}_\Delta(\lambda) S_{11}+ (\mathfrak{q}_1^R(\lambda)+\mathfrak{m}_\Delta(\lambda) \mathfrak{p}_1^R(\lambda)) S_{21}.
\end{equation*}
Since $\mathfrak{q}_1^R(\lambda)=A_1^*{}^{-1}$ (see \eqref{p4_ee1.12}), we obtain the result of the lemma.
\qed \end{proof}

The above lemma allows us to continue $\mathfrak{m}_\Delta$ using the continuation of $m$. We will now establish some results that will allow us to study  zeros and poles of $m$, $\mathfrak{m}_\Delta$, $m-m^\sharp$, and $\mathfrak{m}_\Delta-\mathfrak{m}^\sharp_\Delta$.

\begin{lemma}\label{p4_may2} If $m$ and $\mathfrak{m}_\Delta$ have meromorphic continuations to $\calS_R$ and $\calR_R$, respectively, then for $\lambda\in \pi^{-1}(F_R)$,
\begin{multline}\label{p4_ems}
\left[\mathfrak{m}_\Delta(\lambda)-\mathfrak{m}^\sharp_\Delta(\lambda)\right] (S_{11}+\mathfrak{p}_1(\lambda)S_{21})= \sum_{j=1}^p \left[ m(\widetilde{f}_j(\lambda))-m^\sharp(\widetilde{f}_j(\lambda)) \right] \times \\
\times \left( \begin{array}{cccc} 1&p_1&\cdots & p_{p-1} \\ p_1&p^2_1&\cdots & p_1p_{p-1} \\ \vdots&\vdots&\ddots&\vdots \\ p_{p-1}&p_1p_{p-1} &\cdots & p^2_{p-1} \end{array}
\right) (f_j(\lambda)).
%\mbox{or: } [p_{k-1}(f_j(\lambda))p_{s-1}(f_j(\lambda))]_{k,s=1}^p
\end{multline}
\end{lemma}
\begin{proof}
Immediate from the previous lemma.
\qed
\end{proof}

In order to use equality \eqref{p4_ems}, we will need to understand the detailed behavior of $S_{11}+\mathfrak{p}_1(\lambda)S_{21}$. This is done in Lemma \ref{p4_may5}. To prove it, we will need the following perturbation theory result. For the terminology and basics of perturbation theory, we refer the reader to \cite[Section~2.1]{Kato} or \cite[Section~3.2]{Baumgartel}.

\begin{lemma}\label{perturb}
Let $\mathfrak{U}(z)$ be an  analytic $p\times p$ matrix-valued function in a small neighborhood of $z=z_0$ and $\lambda_0$ be an eigenvalue of $\mathfrak{U}(z_0)$. Suppose that
\begin{list}{}
{
\setlength\labelwidth{2em}%
\setlength\labelsep{0.5em}%
\setlength\leftmargin{2.5em}%
}
  \item[(E1)] The $\lambda_0$-group of perturbed eigenvalues of $\mathfrak{U}(z)$ is $\{ \lambda_1(z), \lambda_2(z), \ldots, \lambda_{2N}(z) \}$, each of multiplicity $1$. Suppose that this $\lambda_0$-group of eigenvalues consists of $N$ cycles $\{\lambda_{2s}(z),\lambda_{2s+1}(z)\}$ of period $2$ $(s=1,2,\ldots,N)$.
  \item[(E2)] The eigenvectors $\vec{g}_{j}(z)$ of $\mathfrak{U}(z)$ corresponding to the eigenvalue $\lambda_j(z)$ can be chosen so that they satisfy
    \begin{align}
    \label{puis1} \vec{g}_{2s}(z) &= \vec{h}_s + (z-z_0)^{1/2} \vec{k}_{2s}  + O(z-z_0),\\
    \label{puis2} \vec{g}_{2s+1}(z) &= \vec{h}_s + (z-z_0)^{1/2} \vec{k}_{2s+1} + O(z-z_0),
    \end{align}
    where $\vec{h}_1,\ldots,\vec{h}_N\in\bbC^p$ are linearly independent constant vectors.
\end{list}

Then the Jordan blocks  corresponding to $\lambda_0$ in the Jordan form of $\mathfrak{U}(z_0)$ are each of size $2\times 2$, and there are $N$ of them.
\end{lemma}
\begin{remarks}
1. Note that given the condition (E1), there always exist eigenvectors having expansions \eqref{puis1}--\eqref{puis2}, with non-zero $\vec{h}_s$ (see Theorem~2 from \cite[Section~6.1.7]{Baumgartel}). What is a non-trivial requirement here is that the vectors $\vec{h}_1,\ldots,\vec{h}_N$ are linearly independent.

2. Condition (E1) easily gives us that the algebraic multiplicity of $\lambda_0$ is $2N$. The condition (E2) says that the geometric multiplicity is \textit{at least} $N$. There doesn't seem to be a general theory that would determine the Jordan form from this.

\end{remarks}
\begin{proof}
Using the perturbation theory on the eigenprojections, we immediately know that the algebraic multiplicity of $\lambda_0$ as an eigenvalue of $\mathfrak{U}(z_0)$ is $2N$. We will now find $N$ linearly independent eigenvectors (which are going to be $\vec{h}_1,\ldots,\vec{h}_N$, of course), and show that each one of them has an associated generalized eigenvector. This determines the Jordan structure we are looking for.

Condition (E1) and the standard perturbation theory tell us that the perturbed eigenvalue functions have the Puiseux expansions
\begin{align}
\label{puis3} \lambda_{2s}(z) &= \lambda_0 + c_s (z-z_0)^{1/2} + O(z-z_0), \\
\label{puis4} \lambda_{2s+1}(z) &= \lambda_0 - c_s (z-z_0)^{1/2} + O(z-z_0),
\end{align}
for any $s=1,\ldots,N$.

Note that $c_s\ne 0$, for otherwise $\{\lambda_{2s}(z),\lambda_{2s+1}(z)\}$ would not constitute a period~$2$ cycle.

%We also need the expansions for eigenvectors. We will make use of Theorem~2 from \cite[Section~6.1.7]{Baumgartel}. Note that because of the multiplicity $1$ assumption of our lemma, each $\lambda_j(z)$-eigenprojection of $\mathfrak{U}(z)$ is of rank $1$ ($m_\alpha=1$ in the notation of \cite[Section~6.1.7]{Baumgartel}). Therefore one can choose eigenvectors $g_{s}(z)$ of $\mathfrak{U}(z)$ corresponding to the eigenvalue $\lambda_s(z)$ which satisfy
%\begin{align}
%\vec{g}_{2s}(z) &= \vec{h}_s + (z-z_0)^{1/2} \vec{k}_{2s}  + O(z-z_0)\\
%\vec{g}_{2s+1}(z) &= \vec{h}_s + (z-z_0)^{1/2} \vec{k}_{2s+1} + O(z-z_0)
%\end{align}
%for any $s=1,\ldots,N$. Note that $\vec{h}_s$ is a non-zero vector by the quoted theorem.

Now taking $z\to z_0$ in $\mathfrak{U}(z) \vec{g}_{2s}(z) = \lambda_{2s}(z) \vec{g}_{2s}(z)$ gives
\begin{equation*}
\mathfrak{U}(z_0) \vec{h}_s = \lambda_0 \vec{h}_s.
\end{equation*}
Similarly, plugging expansions \eqref{puis1}, \eqref{puis2}, \eqref{puis3}, \eqref{puis4} into
$$
\mathfrak{U}(z) \left( \frac{\vec{g}_{2s}(z) - \vec{g}_{2s+1}(z)}{2 c_s (z-z_0)^{1/2}} \right) = \tfrac{1}{2 c_s (z-z_0)^{1/2}} ( \lambda_{2s}(z) \vec{g}_{2s}(z) - \lambda_{2s+1}(z) \vec{g}_{2s+1}(z) ),
$$ and then taking taking $z\to z_0$, gives
\begin{equation*}
\mathfrak{U}(z_0) \left( \frac{\vec{k}_{2s} - \vec{k}_{2s+1}}{2 c_s} \right) = \vec{h}_s + \lambda_0 \frac{\vec{k}_{2s} - \vec{k}_{2s+1}}{2 c_s}.
\end{equation*}
This shows that $\tfrac{\vec{k}_{2s} - \vec{k}_{2s+1}}{2 c_s}$ is a non-zero vector, and, in fact, is the generalized eigenvector associated with $\lambda_0$ and $\vec{h}_s$. Thus we obtain $N$ Jordan blocks of size at least $2$. Since the algebraic multiplicity is $2N$, we obtain the statement of our lemma.
\qed \end{proof}

We can now prove the following lemma. Recall that $\{\gamma_j\}_{j=1}^{p-1}$ are the zeros of the polynomial $\Delta'(z)$. Denote
\begin{equation}\label{funcU}
\mathfrak{U}(\lambda)=S_{11}+\mathfrak{p}_1(\lambda)S_{21}.
\end{equation}

\begin{lemma}\label{p4_may5} The following holds:
\begin{align} \label{p4_ee2.15}
&\det\mathfrak{U}(\lambda)=p^{p} \prod_{j=1}^{p-1} (\lambda-\Delta(\gamma_j)), \\
\label{p4_ee2.16}
&\ker\mathfrak{U}(\lambda)=\spann \{\vec{v}_1(\lambda),\cdots,\vec{v}_p(\lambda)\}^\perp,
\end{align}
where  $\vec{v}_j(\lambda)=(1,p_1(f_j(\lambda)),\cdots,p_{p-1}(f_j(\lambda)))^*$.

In particular, $\mathfrak{U}(\lambda)$ is singular if and only if $\lambda=\Delta(\gamma_j)$, $j=1,\ldots,p-1$, and $\mathfrak{U}(\lambda)^{-1}$ has simple poles at these points.
\end{lemma}
\begin{proof}
Note that by \eqref{p4_ee1.11},
$$
\mathfrak{U}(\lambda) = S_{11}+(\lambda\bdone-B_{1}) A_{1}^*{}^{-1}S_{21}=S_{11}+(\lambda\bdone-T_{11}) T_{21}^{-1}S_{21},
$$
where $S_{ij}$ and $T_{ij}$ are the $p\times p$ blocks of $\Delta'(\calJ)$ and $\Delta(\calJ)$, respectively.

Take any $\mu\in\bbC$, and let
\begin{equation*}
\begin{aligned}
\widehat{u}(\mu)&=(1,p_1(\mu),\ldots,p_{j}(\mu),\ldots)^*, \\
\vec{u}_1(\mu)&=(1,p_1(\mu),\ldots,p_{p-1}(\mu))^*, \\
\vec{u}_2(\mu)&=(p_{p}(\mu),p_{p+1}(\mu),\ldots,p_{2p-1}(\mu))^*.
\end{aligned}
\end{equation*}
Then $\widehat{u}^* \calJ={\mu }\widehat{u}^*$ in the formal sense (note that $\widehat{u}\notin \ell^2$ in general). This gives $\widehat{u}^* \Delta(\calJ)=\Delta({\mu}) \widehat{u}^*$ and $\widehat{u}^* \Delta'(\calJ)=\Delta'({\mu}) \widehat{u}^*$ in the formal sense. However $\Delta(\calJ)$ and $\Delta'(\calJ)$ are banded matrices, so we can conclude that
\begin{align*}
\vec{u}_1^* T_{11}+\vec{u}_2^* T_{21}&=\Delta({\mu}) \vec{u}_1^*, \\
\vec{u}_1^* S_{11}+\vec{u}_2^* S_{21}&=\Delta'({\mu}) \vec{u}_1^*.
\end{align*}
The first equality implies $\vec{u}_1^* T_{11} T_{21}^{-1}+\vec{u}_2^* =\Delta({\mu}) \vec{u}_1^* T_{21}^{-1}$, and therefore
\begin{equation*}
\vec{u}_1^*[S_{11}+ (\lambda-T_{11})T_{21}^{-1} S_{21}]=\vec{u}_1^* [\Delta'({\mu})+(\lambda-\Delta({\mu}))T_{21}^{-1}S_{21}].
\end{equation*}

This shows that if $\lambda=\Delta({\mu})$, then $\mathfrak{U}(\lambda)^* \vec{u}_1(\mu) = \Delta'({\bar\mu}) \vec{u}_1(\mu)$, or equivalently, \begin{equation}\label{eigen}
\mathfrak{U}(\bar\lambda)^* \vec{u}_1(\bar \mu) = \Delta'({\mu}) \vec{u}_1(\bar\mu).
\end{equation}
Since $\mathfrak{U}(\bar\lambda)^*$ is a $p\times p$ matrix, we now know its spectrum:
\begin{equation}\label{spectrum}
\sigma(\mathfrak{U}(\bar\lambda)^*)=\{\Delta'(\Delta^{-1}(\lambda))\} = \{ \Delta'(f_1(\lambda)), \ldots, \Delta'(f_p(\lambda))\}.
\end{equation}
Indeed, by perturbation theory this equality is true even if some of the points $\{ \Delta'(f_1(\lambda)), \ldots, \Delta'(f_p(\lambda))\}$ coincide. Note that this happens if and only if $\lambda=\Delta(\gamma_j)$ for some $j$.

Thus the characteristic polynomial of $\mathfrak{U}(\bar\lambda)^*$ is
$$
\det(\mathfrak{U}(\bar\lambda)^*-t \bdone) = \prod_{k=1}^p (\Delta'(f_k(\lambda))-t).
$$
At $t=0$, using \eqref{disc_fact}:
\begin{equation}\label{d_prime}
\begin{aligned}
\det \mathfrak{U}(\lambda) = \det \mathfrak{U}(\bar\lambda)^* = \prod_{k=1}^p \Delta'(f_k(\lambda)) &= \prod_{k=1}^p pc_0\prod_{j=1}^{p-1} (f_k(\lambda)-\gamma_j)
%= \prod_{j=1}^{p-1} \prod_{k=1}^p (f_k(\lambda)-\gamma_j)
\\
&= p^{p} \prod_{j=1}^{p-1} (\lambda-\Delta(\gamma_j)).
\end{aligned}
\end{equation}
This establishes \eqref{p4_ee2.15} (alternatively one can directly see that \eqref{spectrum} contains zero if and only $\lambda=\Delta(\gamma_j)$ for some $j$, and then count the degree of the polynomials).

%This shows that $\vec{u}_1(\mu)$ is an eigenvector of $[S_{11}+ (\lambda-T_{11})T_{21}^{-1} S_{21}]^*$ if $\lambda=\Delta({\mu})$, and it is actually in the kernel if $\Delta'(\mu)=0$, i.e., $\mu\in \{\gamma_j\}_{j=1}^{p-1}$. Note that $T_{21}^{-1}S_{21}$ is a matrix with $0$'s on and below the main diagonal with positive elements right above it, which  implies that the degree of the polynomial $\det(S_{11}+(\lambda\bdone-T_{11}) T_{21}^{-1}S_{21})$ is $p-1$. This establishes \eqref{p4_ee2.15}.

Note that the system of vectors $\{(1,p_1(z_j),\cdots,p_{p-1}(z_j))\}_{j=1}^k$ is linearly independent if and only if all the points $z_j$ are distinct: easy use of Vandermonde determinant and the fact that $p_n$ is of degree $n$. Therefore if $\lambda\notin \{ \Delta(\gamma_j) \}_{j=1}^{p-1}$, then vectors $\vec{v}_j(\lambda)=(1,p_1(f_j(\lambda)),\cdots,p_{p-1}(f_j(\lambda)))^*$ form a basis of $\bbC^p$, and \eqref{p4_ee2.16} is trivial.

Suppose $\lambda_0=\Delta(\gamma_k)$ for some $k$. We showed in \eqref{eigen} that each $\vec{v}_j(\bar\lambda_0)$, $1\le j\le p$, is an eigenvector of $\mathfrak{U}(\bar\lambda_0)^*$ with eigenvalue $\Delta'(f_j(\lambda_0))$. Now let us apply Lemma~\ref{perturb} to $\mathfrak{U}(\bar\lambda)^*$ around the point $\lambda_0$.  Note that in place of the Lemma's $\mathfrak{U}(z)$, $z_0$, $\lambda_0$, we feed $\mathfrak{U}(\bar\lambda)^*$, $\lambda_0$, $0$, respectively, hoping it will not cause confusion. Note that  if $\Delta'(f_j(\lambda_0))=0$ for some $j$, then $f_j(\lambda)$ is one of the two branches of a multivalued analytic function with branching degree $2$ around $\lambda_0$. This is because each $\gamma_j$ is a \textit{simple} zero of $\Delta'$ (follows from Lemma~\ref{discriminant}(i)). Thus (E1) of Lemma~\ref{perturb} is satisfied. The linear independence in (E2) of Lemma~\ref{perturb} follows from the above-mentioned fact that the system $\{(1,p_1(z_j),\cdots,p_{p-1}(z_j))\}_{j=1}^k$ is linearly independent if and only if all the points $z_j$ are distinct.

Therefore we can conclude that the Jordan form of $\mathfrak{U}(\bar\lambda_0)^*$ consists of $1\times 1$ blocks corresponding to non-zero eigenvalues and $2\times 2$ blocks corresponding to zero eigenvalues. It is clear that for such matrices the range is precisely equal to the span of all eigenvectors. Thus,
\begin{equation*}
\ran \mathfrak{U}(\bar\lambda_0)^* = \spann \{\vec{v}_1(\bar\lambda_0),\cdots,\vec{v}_p(\bar\lambda_0)\}
\end{equation*}
holds. This implies \eqref{p4_ee2.16}, since $\ker \mathfrak{U}(\bar\lambda_0)=(\ran \mathfrak{U}(\bar\lambda_0)^*)^\perp$ and $\lambda_0=\bar\lambda_0$.

Finally, the poles of $\mathfrak{U}(\lambda)^{-1}$ at $\lambda=\Delta(\gamma_j)$ are simple by \eqref{p4_ee2.15}, \eqref{p4_ee2.16}, and Lemma~\ref{p4_lm0}.
\qed \end{proof}
%\begin{remark}
%It is clear from the proof that Lemma \ref{p4_may5} is a just a special case of the following fact: for any polynomials $r_1, r_2$ of degrees $k_1>k_2$, $\det(S_{11}+(\lambda\bdone-T_{11}) T_{21}^{-1}S_{21})=c \prod_{j=1}^{k_2} (\lambda-r_1(\zeta_j))$, where $S_{ij}$ and $T_{ij}$ are the $k_1\times k_1$ blocks of $r_2(\calJ)$ and $r_1(\calJ)$, respectively, and $\zeta_j$ are the zeros of $r_2$.
%\end{remark}

We now know everything we need about $S_{11}+\mathfrak{p}_1(\lambda)S_{21}$. We also need to analyze the right-hand side of \eqref{p4_ems}. Let us assign it a name:
\begin{equation}\label{funcL}
\mathfrak{L}(\lambda)= \sum_{j=1}^p \left[ m(\widetilde{f}_j(\lambda))-m^\sharp(\widetilde{f}_j(\lambda)) \right]
\left( \begin{array}{cccc} 1&p_1&\cdots & p_{p-1} \\ p_1&p^2_1&\cdots & p_1p_{p-1} \\ \vdots&\vdots&\ddots&\vdots \\ p_{p-1}&p_1p_{p-1} &\cdots & p^2_{p-1} \end{array}
\right) (f_j(\lambda)).
\end{equation}

\begin{lemma}\label{p4_may6} The following holds:
$$
\det\mathfrak{L}(\lambda)=\left(\tfrac{p^{p}}{c_0^p} \prod_{j=1}^{p-1} a_j^{-2(p-j)} \right) \prod_{j=1}^p \left( m(\widetilde{f}_j(\lambda))-m^\sharp(\widetilde{f}_j(\lambda)) \right) \prod_{j=1}^{p-1} (\lambda-\Delta(\gamma_j)) .
$$
If $\lambda=\widetilde\Delta(\gamma_k)$ and all $m(\widetilde{f}_j(\lambda))-m^\sharp(\widetilde{f}_j(\lambda))$ are regular and non-zero, then
\begin{equation}\label{kerL}
\ker\mathfrak{L}(\lambda)= \spann \{\vec{v}_1(\lambda),\cdots,\vec{v}_p(\lambda)\}^\perp,
\end{equation}
where $\vec{v}_j(\lambda)=(1,p_1(f_j(\lambda)),\cdots,p_{p-1}(f_j(\lambda)))^*$, and $\mathfrak{L}(\lambda)^{-1}$ has simple poles at these points.
\end{lemma}

\begin{proof}
Let $\eta_j=m(\widetilde{f}_j(\lambda))-m^\sharp(\widetilde{f}_j(\lambda))$. Then %determinant $\det\mathfrak{L}(\lambda)$ can be computed as follows:
\begin{align*}\label{1111}
\det\mathfrak{L}(\lambda) &= \det \sum_{j=1}^p \eta_j \, [p_{k-1}(f_j(\lambda))p_{s-1}(f_j(\lambda))]_{k,s=1}^p \\
 &= \det \left[\sum_{j=1}^p \eta_j p_{k-1}(f_j(\lambda))p_{s-1}(f_j(\lambda))\right]_{k,s=1}^p
\\
 &=\det\left( \left[ \eta_j p_{k-1}(f_j(\lambda))\right]_{k,j=1}^p  \left[  p_{s-1}(f_j(\lambda))\right]_{j,s=1}^p \right) \\
 &=  \left(\det \left[  p_{s-1}(f_j(\lambda))\right]_{j,s=1}^p \right)^2 \prod_{j=1}^p \eta_j.
\end{align*}
Since $p_j$ is of degree $j$, by performing elementary row operations we can reduce $\det \left[  p_{s-1}(f_j(\lambda))\right]_{j,s=1}^p $ to the Vandermonde determinant times the product of the leading coefficients of $1,p_1,\ldots,p_{p-1}$. Using \eqref{leading}, we get
\begin{align*}
\det\mathfrak{L}(\lambda) &= \prod_{j=1}^p \eta_j \prod_{j=1}^{p-1} (a_1\ldots a_j)^{-2} \prod_{j<s} (f_j(\lambda)-f_s(\lambda))^2 \\
&= \prod_{j=1}^p \eta_j \prod_{j=1}^{p-1} a_j^{-2(p-j)}
\prod_{j=1}^p \mathop{\prod_{s=1}^p}_{s\ne j} (f_j(\lambda)-f_s(\lambda)).
\end{align*}
Now observe that $$\mathop{\prod_{s=1}^p}_{s\ne j} (f_j(\lambda)-f_s(\lambda))=\tfrac{1}{c_0} \Delta'(f_j(\lambda))$$ by \eqref{p4_ee2.13}, and so the determinant is equal to
\begin{equation*}
\tfrac{1}{c_0^p} \prod_{j=1}^p \eta_j \prod_{j=1}^{p-1} a_j^{-2(p-j)} \prod_{j=1}^p \Delta'(f_j(\lambda))=  \tfrac{p^{p}}{c_0^p} \prod_{j=1}^{p-1} a_j^{-2(p-j)} \prod_{j=1}^p \eta_j \prod_{s=1}^{p-1}(\lambda-\Delta(\gamma_s)),
\end{equation*}
where in the last step we reused the computations from \eqref{d_prime}. This proves the first statement of the lemma.

Suppose that $\lambda=\Delta(\gamma_k)$. That any vector orthogonal to $\{\vec{v}_1(\lambda),\cdots,\vec{v}_p(\lambda)\}$ must be in the kernel is clear, since the $j$-th row of the matrix in \eqref{p4_ems} is obtained from its first row by multiplication by $p_{j-1}$. Therefore
\begin{equation}\label{supset}
\ker\mathfrak{L}(\lambda)\supseteq \spann \{\vec{v}_1(\lambda),\cdots,\vec{v}_p(\lambda)\}^\perp.
\end{equation}
Note that $\dim\ker\mathfrak{L}(\lambda)$ is less than or equal to the order of $\lambda$ as the root of $\det\mathfrak{L}(\lambda)$ (it could be strictly less if one of the $\kappa$'s is $\ge2$ in Lemma~\ref{p4_lmSM}). But this order is precisely equal to $p$ minus the cardinality of $\{ f_1(\lambda ),\ldots,f_p(\lambda) \}$. This implies that
\begin{equation}\label{subset}
\dim\ker\mathfrak{L}(\lambda) \le \dim \spann \{\vec{v}_1(\lambda),\cdots,\vec{v}_p(\lambda)\}^\perp.
\end{equation}
But then \eqref{supset} and \eqref{subset} imply that $\ker\mathfrak{L}(\lambda)= \spann \{\vec{v}_1(\lambda),\cdots,\vec{v}_p(\lambda)\}^\perp$.

Finally, each zero of $\mathfrak{L}(\lambda)$ is simple by Lemma \ref{p4_lm0}.
\qed \end{proof}

\begin{lemma}\label{p4_may3} If $m$ and $\mathfrak{m}_\Delta$ have meromorphic continuations to $\calS_R$ and $\calR_R$, respectively, then for $\lambda\in \pi^{-1}(F_R)$,
\begin{equation*}\label{p4_dets}
\det\left(\mathfrak{m}_\Delta(\lambda)-\mathfrak{m}^\sharp_\Delta(\lambda)\right)=\left(\tfrac{1}{c_0^p}\prod_{j=1}^p a_j^{-2(p-j)}\right) \prod_{j=1}^p \left( m(\widetilde{f}_j(\lambda))-m^\sharp(\widetilde{f}_j(\lambda)) \right).
\end{equation*}
\end{lemma}
\begin{remark}
Note that if we take $\lambda \in\pi^{-1}(\fre)$ in the lemma, then we can recover the formula from Damanik--Killip--Simon relating the determinant of the density $\frac{d\mu_\Delta}{dx}$ of $\mu_\Delta$ and the density $\frac{d\mu}{dx}$ of $\mu$ (see \cite[Prop 11.1]{DKS}). In our notation it looks as follows:
\begin{equation*}
\det \left[\frac{d\mu_\Delta(\lambda)}{d\lambda} \right]= \tfrac{1}{c_0^p}\prod_{j=1}^{p-1} a_j^{-2(p-j)} \prod_{j=1}^p \frac{d\mu}{dx}(f_j(\lambda)).
\end{equation*}
\end{remark}
\begin{proof}
Immediate from Lemmas \ref{p4_may2}, \ref{p4_may5}, and \ref{p4_may6}.
\qed \end{proof}

The next lemma will allow us to assume that $m$ satisfies conditions (P1) and (P2) stated below, which will considerably simplify the proof of the main results.

Recall that $\{\xi_j\}_{j=1}^N$ are all of the preimages $\Delta^{-1}(\Delta(\gamma_j))$, where $\gamma_j$ are the critical points of $\Delta'$.

\begin{lemma}\label{p4_may4}
Let $a_{0}>0, b_{0}\in\bbR$, and let $\calJ^{(-1)}=(a_n,b_n)_{n=0}^\infty$ be the Jacobi matrix obtained from $\calJ=(a_n,b_n)_{n=1}^\infty$ by adding one column and one row with the coefficients $a_0,b_0$. Let $m$ and $m^{(-1)}$ be the Borel transforms of the spectral measure of $\calJ$ and $\calJ^{(-1)}$, respectively. If $m$ satisfies (ii) of Theorem \ref{p4_th_merom}/\ref{p4_th_finite_merom}, then so does $m^{(-1)}$.

Moreover, for any $\veps>0$, there exist $a_0,b_0,a_{-1},b_{-1}$ such that the Jacobi matrix $\calJ^{(-2)}=(a_{n},b_{n})_{n=-1}^\infty$ (with two rows and columns added) satisfies
\begin{list}{}
{
\setlength\labelwidth{2em}%
\setlength\labelsep{0.5em}%
\setlength\leftmargin{2.5em}%
}
\item[(P1)] $m^{(-2)}$ does not have poles at any $(\xi_j)_\pm$ and band edges;
\item[(P2)] for any two poles $\zeta_1$, $\zeta_2$ of $m^{(-2)}$ in $\calS_{R-\veps}$, %\setminus\{\alpha_j,\beta_j\}_{j=1}^p$
  $\widetilde{\Delta}(\zeta_1)\ne\widetilde{\Delta}(\zeta_2)$, $\widetilde{\Delta}(\zeta^\sharp_1)\ne\widetilde{\Delta}(\zeta_2)$.
\end{list}
\end{lemma}
%\begin{remark} It seems $k=1$ should be sufficient, but would overcomplicate the proof for no reason.
%\end{remark}
\begin{proof}
Suppose that $m$ satisfies (ii) of Theorem \ref{p4_th_merom}. By the recursion
\begin{equation}\label{p4_m-recur}
a_{0}^2 m(z)=-z+b_{0}-{m^{(-1)}(z)}^{-1}
\end{equation}
we can extend $m^{(-1)}$ to the same domain as $m$. So $m^{(-1)}$ satisfies (ii)(a) of Theorem~\ref{p4_th_merom}. Moreover,
\begin{equation}\label{p4_m-m}
m(z)-m^\sharp(z)=\frac{m^{(-1)}(z)-m^{(-1)}{}^\sharp(z)}{a_0^2 m^{(-1)}(z)m^{(-1)}{}^\sharp(z)} \quad \mbox{ for } z\in \pi^{-1}(E_R).
\end{equation}

Assume $m^{(-1)}(z)$  has a pole at a point in $\pi^{-1}(\fre\setminus\cup_{j=1}^p \{\alpha_j,\beta_j\})$. Then \eqref{p4_m-recur} implies that  $m$ is real at this point, which violates (ii)(c). Assume $m^{(-1)}(z)$  has a pole of order $k\ge2$ at a band edge $z\in\pi^{-1}(\cup_{j=1}^p \{\alpha_j,\beta_j\})$. Then $m^{(-1)}{}^\sharp(z)$ has the same order pole at this point, and therefore $m^{(-1)}(z)-m^{(-1)}{}^\sharp(z)$ has a pole of order at most $k$. Now \eqref{p4_m-m} implies that $m-m^\sharp$ has a zero of order at least $2k-k\ge2$, contradicting (ii)(c) for $m$. Thus $m^{(-1)}$ satisfies (ii)(b).

Assume $m^{(-1)}(z)$ and $m^{(-1)}{}^\sharp(z)$ are both regular and $m^{(-1)}(z)-m^{(-1)}{}^\sharp(z)=0$, for some $z$ not at a band edge. Then \eqref{p4_m-m} implies that $m$ violates (ii)(c) or (ii)(d) of Theorem \ref{p4_th_merom}, a contradiction. Thus $m^{(-1)}$ satisfies (ii)(c) for $z$ not at a band edge.

Let us verify (ii)(c) for $m^{(-1)}$ at a band edge.

Assume $m^{(-1)}$ is finite and non-zero at a band edge. Then so is $m^{(-1)}{}^\sharp$, and then $m^{(-1)}-m^{(-1)}{}^\sharp = a_0^2(m-m^\sharp)m^{(-1)} m^{(-1)}{}^\sharp$ has at most first order pole there.

Now let $m^{(-1)}$ have a zero of order $k\ge1$ at $z_0\in\pi^{-1}(\cup_{j=1}^p \{\alpha_j,\beta_j\})$. Then \eqref{p4_m-recur} shows that necessarily $k=1$. This means that locally around $z_0$,
$$
m^{(-1)}(z)=s_1{\sqrt{z-z_0}} + s_2 (z-z_0) + O(z-z_0)^{3/2}
$$
for a non-zero constant $s_1$.
But then $m^{(-1)}{}^\sharp(z)=-s_1{\sqrt{z-z_0}} + s_2 (z-z_0) - O(z-z_0)^{3/2}$, and so $m^{(-1)}(z)-m^{(-1)}{}^\sharp(z) = 2s_1{\sqrt{z-z_0}} + O(z-z_0)^{3/2}$ has first order zero too.

Lastly, assume $m^{(-1)}$ has a pole at a band edge. We showed that then this pole is simple. Again, $m^{(-1)}{}^\sharp$ has a first order pole with the  coefficient near $\tfrac{1}{\sqrt{z-z_0}}$ being negative to that of $m^{(-1)}$. Therefore $m^{(-1)}-m^{(-1)}{}^\sharp$ still has a first order pole. Thus is does not vanish, and so $m^{(-1)}$ satisfies (ii)(c).

%Now note that if $m$ has a zero/pole of order $k\ne0$ at a band edge, then so does $m^\sharp$ with the leading coefficient being negative of  that of $m$ ($\sqrt{\Delta^2-4}$ changes sign when we change sheets). Therefore $m-m^\sharp$ also has a zero/pole of order $k$. This and \eqref{p4_m-m} shows that if $m^{(-1)}$ has a zero/pole of order $k\ne0$ there, then $m$ has a zero/pole of order $-k\ne0$. But (ii)(c)--(ii)(d) for $m$ say that $|k|$ can be at most $1$. Thus $m^{(-1)}$ satisfies (ii)(b)--(ii)(c) for the band edge points too.

Finally, let us check (ii)(d) for $m^{(-1)}$. Assume $m^{(-1)}(z)$ and $m^{(-1)}{}^\sharp(z)$ both have a pole at $z_0\in\pi^{-1}(E_R\setminus\fre)$. Then by \eqref{p4_m-recur} $m(z_0)=m^\sharp(z_0)=\frac{b_0-\pi(z_0)}{a_0^2}$, which means that $m$ violates (ii)(c). Contradiction.

\smallskip

Let us prove the ``moreover'' part of the lemma now. Note that all the poles of $m^{(-1)}$ occur at the points where $a_0^2 m(z)=b_0-z$. Denote the finite number of distinct poles of $m$ in $\calS_{R-\veps}$ by $\{z_j\}_{j=1}^K$. Let $M_1=\max_j |\pi(z_j)|$. Choose small $\delta>0$ such that the $\delta$-neighborhoods $U_\delta(z_j)=\{z: |z-z_j|<\delta\}$ of these points are disjoint and lie inside $\calS_{R-\veps}$. Let
$$
M_2 =\sup_{z\in \calS_{R-\veps}\setminus \cup_{j=1}^K U_\delta(z_j)} |m(z)|.
$$
%to be larger than the supremum of $|m(z)|$ over all $z$ in $\calS_{R-\veps}$ not in these neighborhoods.

Let $b_0(t)=M_1+M_2+t$ for $t \gg 0$, and choose any $a_0$ satisfying $0<a_0<\sqrt{\delta/M_2}$. For each such $a_0, b_0(t)$ let $m^{(-1)}(a_0,b_0(t))$ be the $m$-function of $\calJ^{(-1)}=(a_n,b_n)_{n=0}^\infty$. Note that if $z$ is not in one of $U_\delta(z_j)$ or $U_\delta(b_0(t)_+)$, then $z$ cannot be a pole of $m^{(-1)}(a_0,b_0(t))$. Indeed, for such $z$, $|a_0^2 m(z)| \le \delta < |b_0-z|$. Note that for large $t$, $b_0(t)_-$ is not in $\calS_{R-\veps}$, and we can ignore $U_\delta(b_0(t)_-)$.

Let $a_0^2 m(z)+z$ around each $z_j$ be locally $k_j$-to-$1$ (where $k_j \ge1$ is the order of the pole of $m$ at $z_j$). Therefore assuming $t$ is large enough, we will have precisely $k_j$ distinct solutions to $a_0^2 m(z)+z=b_0(t)$ in each $U_\delta(z_j)$%(they are distinct since $z_j$ itself cannot be a solution)
, i.e., there are precisely $k_j$ distinct first order poles of $m^{(-1)}(a_0,b_0(t))$ in each $U_\delta(z_j)$.

For large enough $t$ there will be exactly one solution to $a_0^2 m(z)=b_0(t)-z$ in $U_\delta(b_0(t)_+)$. Indeed,  $m$ is monotonically increasing to zero as $\bbR\ni z\to+\infty$ (see \eqref{intro1}). Therefore for large $t$, $a_0^2 m(z)=b_0(t)-z$ will have exactly one real solution in $U_\delta(b_0(t)_+)$. Since any pole of $m^{(-1)}$ on $\calS_+$ must be real, there must be a unique pole of $m^{(-1)}$ in $U_\delta(b_0(t)_+)$.

Thus there are precisely $1+\sum_{j=1}^K k_j$ first order poles of $m^{(-1)}(a_0,b_0(t))$ in $\calS_{R-\veps}$, which are distinct for any $t$ large enough. Denote the locations of these poles by $\zeta_j(t)$ (note that each $\zeta_j(t)$ is a continuous function).

The restriction (P1) requires only $b_0(t)\ne a_0^2 m((\xi_j)_\pm)+\xi_j$, and $b_0(t)\ne a_0^2 m(\alpha_j)+\alpha_j$, $b_0(t)\ne a_0^2 m(\beta_j)+\beta_j$, which excludes only a finite number of allowable $b_0(t)$ for each $a_0$.

Thus choosing $t$ large enough will always ensure that $m^{(-1)}(a_0,b_0(t))$ satisfies (P1) and has only first order poles in $\calS_{R-\veps}$.
%any other $b_0(t)$ therefore produces $\calJ^{(-1)}$ satisfying ($\dagger$) and having only first order poles in $\calS_{R-\veps}$. %Notice also that the poles $\{z_j\}$ of $\calJ^{(-1)}$ cannot lie in $\pi^{-1}(\fre)$ (no eigenvalues in the interior of $\fre$ by (ii)(b), and we arranged that $m^{(-1)}$ has no first-order pole at the band edges too).
Performing this procedure and renaming $m^{(-1)}$ to $m$, we may now assume that $m$ already satisfies (P1) and has only first order poles in $\calS_{R-\veps}$.%, and has no poles in $\pi^{-1}(\fre)$.

In particular, since each pole of $m$ is assumed to be simple, $k_j=1$. Therefore there are precisely $K+1$ poles of  $m^{(-1)}$ in $\calS_{R-\veps}$: one pole $\zeta_j(t)$ in each $U_\delta(z_j)$, $1\le j\le K$, and one pole $\zeta_{K+1}(t)$ in $U_\delta(b_0(t)_+)$.

Choose any pair of indices $1\le j,n\le K+1$. We will be checking which $a_0$ and $t$ would make $m^{(-1)}(a_0,b_0(t))$ satisfy (P2) for the pair of poles $\zeta_j(t), \zeta_n(t)$.

First observe that $\zeta_{K+1}(t)\to\infty_+$ when $t\to\infty$, while $\zeta_j(t)\in U_\delta(z_j)$, so if $t$ is large enough then $\zeta_{K+1}(t)$ cannot cause any trouble with respect to (P2).

Now note that if $\delta$ is small enough, then
\begin{align*}
\widetilde{\Delta}(z_j)\ne \widetilde{\Delta}(z_n) & \Rightarrow  \widetilde{\Delta}(U_\delta(z_j))\cap \widetilde{\Delta}(U_\delta(z_n))=\varnothing ,
\\
\widetilde{\Delta}(z^\sharp_j)\ne \widetilde{\Delta}(z_n) & \Rightarrow \widetilde{\Delta}(U_\delta(z_j)^\sharp)\cap \widetilde{\Delta}(U_\delta(z_n))=\varnothing.
\end{align*}
So if $m$ satisfies (P2) some $z_j$, $z_n$, then $m^{(-1)}$ satisfies (P2) for the corresponding poles $\zeta_j(t)$, $\zeta_n(t)$.

Assume that $m$ does not satisfy (P2), say for the poles $z_1$ and $z_2$. Without loss of generality we may assume
$\widetilde{\Delta}(z_1)=\widetilde{\Delta}(z_2)\equiv\lambda_0$ (the case $\widetilde{\Delta}(z^\sharp_1)=\widetilde{\Delta}(z_2)$ can be treated in the same way).

%If we pick the coefficients $a_0$ and $b_0(t)$ so that $m^{(-1)}(a_0, b_0(t))$ satisfies ($\ddagger$) for the corresponding poles in $U_\delta(z_1)$ and $U_\delta(z_2)$, then we can keep repeating this procedure to get rid of all ``resonances''.

Fix any $a_0$ ($0<a_0<\sqrt{\delta/M_2}$). Suppose that $m^{(-1)}(a_0, b_0(t))$ fails the condition (P2) for uncountably many $t$ at $\zeta_1(t)$ and $\zeta_2(t)$. Recall that $\zeta_1(t)$, $\zeta_2(t)$ are the unique solutions of $a_0^2 m(z)=b_0(t)-z$ in $U_\delta(z_1)$, $U_\delta(z_2)$, respectively. This implies
$\widetilde{\Delta}(\zeta_1(t))=\widetilde{\Delta}(\zeta_2(t))=:\lambda(t)$. This means that we can choose \textit{different} branches $\widetilde{f}_1, \widetilde{f}_2$ of $\widetilde{\Delta}^{-1}$ around $\lambda_0$ (note that $\pi(\lambda_0)$ is not critical point of $\Delta$ since (P1) holds for $m$), such that $\zeta_1(t)=\widetilde{f}_1(\lambda(t))$, $\zeta_2(t)=\widetilde{f}_2(\lambda(t))$. This implies
$$
a_0^2 ( m(\widetilde{f}_1(\lambda(t)))-m(\widetilde{f}_2(\lambda(t))) ) =  f_2(\lambda(t))- f_1(\lambda(t)).
$$
for uncountably many $t$. But then by analytic continuation we obtain
$$
a_0^2 ( m(\widetilde{f}_1(\lambda))-m(\widetilde{f}_2(\lambda)) ) = f_2(\lambda)- f_1(\lambda).
$$
for all $\lambda$ in a neighborhood of $\lambda_0$. This may, in fact, happen. However then any $a_0$ different from the chosen one would violate this condition. This means that there may be only one $a_0$ for which $\widetilde{\Delta}(\zeta_1(t))=\widetilde{\Delta}(\zeta_2(t))$ holds for uncountably many $t$. Every other $a_0$ will have at most countably many exceptions. Note, in particular, that this allows us to take $t$ as large as we need (which is important in regards to $\zeta_{K+1}(t)$, as well as to make sure that $m^{(-1)}$ still satisfies (P1)).

Since there are finitely many pairs of indices $1\le j,n\le K+1$, we can conclude that there exists a choice of $a_0$ and $t$ which works for all of them, i.e., $m^{(-1)}$ satisfies $(P2)$. Moreover, $t$ can be chosen large enough, so that $m^{(-1)}$ still satisfies $(P1)$.
%However choose now any $0<a'_0<a_0$. Either we get rid of the resonance for this $a'_0$ and some $b_0(t)$, or we again obtain
%$$
%{a'_0}^2 ( m(\widetilde{f}_1(\lambda))-m(\widetilde{f}_2(\lambda)) ) = f_2(\lambda)- f_1(\lambda)
%$$
%for all $\lambda$ in a neighborhood of $\lambda_0$. The last two equalities imply $f_2(\lambda)-f_1(\lambda)\equiv 0$ giving the contradiction.
\qed \end{proof}
%\begin{remark}
%If $m$ satisfies Theorem \ref{p4_th_finite_merom}, then the ``moreover'' part of Lemma \ref{p4_may4} reads as ``for any $R>0$ \ldots '' and condition ($\ddagger$) is satisfies for any two poles in $\calS_R$.
%\end{remark}

Finally, we will need the following result, which is the analogue of Lemma \ref{p4_convergence}.
\begin{lemma}\label{p4_convergence2}
Assume that $\calJ=(a_n,b_n)_{n=1}^\infty$ satisfies
\begin{equation*}
\limsup_{n\to\infty} \left(|a_n-a_n^{(0)}| + |b_n-b_n^{(0)}|\right)^{1/2n}  \le R^{-1},
\end{equation*}
where $\calJ^0=\big(a_n^{(0)},b_n^{(0)}\big)_{n=1}^\infty$ is a $p$-periodic Jacobi matrix in $\calT_\fre$.
Let $m$, $m^{(n)}$, and $m^0$ be the Borel transform of the spectral measure of $\calJ,\calJ^{(n)}$, and $\calJ^0$, respectively. Suppose that $m$ has a meromorphic continuation to $\calS_R$. Then $m^{(np)}(z)\to m^0(z)$ as $n\to\infty$ for any $z\in \calS_R$.
\end{lemma}
\begin{remark}
In fact, if we view $m$ as a function $\calS_R \to \bbC\cup\{\infty\}$, then the convergence is uniform on compacts {with respect to the spherical distance} on the Riemann sphere $\bbC\cup\{\infty\}$ .
\end{remark}
\begin{proof}
Since $(\calJ^0)^{(np)}=\calJ^0$, we have $\calJ^{(np)} \to \calJ^0$ in norm. This also gives us $\Delta(\calJ^{(np)}) \to \Delta(\calJ^0)$.  Note that convergence in norm implies convergence of the resolvents, which gives us $m^{(np)}(z)\to m^0(z)$, \textit{but only for $z\in\calS_+$}.

Fix any point $z\in\calS_R$. For this lemma only, let us employ the following convention. For any \textit{scalar} Jacobi matrix $\mathcal{I}$, let us write $m(\mathcal{I})$ to mean the $m$-function (i.e., the Borel transform of the spectral measure) of $\mathcal{I}$ evaluated at $z\in\calS_R$, the dependence on which we will omit for convenience. For any \textit{block} Jacobi matrix $\mathcal{I}$, let us write $\mathfrak{m}_\Delta(\mathcal{I})$ to mean the (matrix-valued) $m$-function of $\mathcal{I}$ evaluated at $\widetilde{\Delta}(z)\in\calR_R$. %(where one takes $\Delta(z)\in \calR_+$ if $z\in\calS_+$ and $\Delta(z)\in \calR_-$ if $z\in\calS_-$). %The matrix-valued $m_\Delta(\calJ)$ will of course run us into trouble when we hit a pole of one of these $m$-functions, but since we look at limits, and these limits will happen to be finite, we are fine.

Let also $m^0$ be the $m$-function of $\calJ^0$, evaluated at $z$, and $\mathfrak{m}_\Delta^0$ be the $m$-function of the free block Jacobi matrix ($\frac{-\lambda\pm\sqrt{\lambda^2-4}}{2}\bdone$), evaluated at $\widetilde{\Delta}(z)$.

%We, of course, want to show $m(\calJ^{(np)})\to m(\calJ^0)$. This holds trivially for $z\in\calS_+$, but for $\calS_-$ we need additional arguments.

Let us write \eqref{p4_func} as $m(\calJ)=g(\mathfrak{m}_\Delta(\Delta(\calJ)),\{a_j\}_{j=1}^N,\{b_j\}_{j=1}^N)$, where $g$ is a continuous function that takes one $p\times p$ matrix-valued parameter and $2N$ real parameters. %(here $\{a_j\}_{j=1}^N,\{b_j\}_{j=1}^N$ are the first Jacobi parameters of the matrix $\calJ$).
Indeed, the right-hand side of \eqref{p4_func} depends on $\mathfrak{m}_\Delta(\Delta(\calJ))$, the first orthogonal polynomial $\mathfrak{p}_1$ of $\Delta(\calJ)$, and the first column of the product $\prod_{j\ne l} (\calJ- f_j(\Delta(z)))$. The latter two objects are smooth functions (in fact, polynomials) of first $N$ Jacobi parameters $\{a_j\}_{j=1}^N,\{b_j\}_{j=1}^N$ of $\calJ$, for $N$ sufficiently large but finite. This proves that if $\calJ_k\to \calJ$ and $\mathfrak{m}_\Delta(\Delta(\calJ_k))\to \mathfrak{m}_\Delta(\Delta(\calJ))$ then $m(\calJ_k)\to m(\calJ)$.

By Lemma \ref{p4_convergence} we have that $\mathfrak{m}_\Delta\big(\Delta(\calJ)^{(n)}\big) \to \mathfrak{m}_\Delta^0$.
Note that  $\Delta(\calJ^{(np)})\ne \Delta(\calJ)^{(n)}$. However, $\Delta(\calJ^{(np)})$ and $\Delta(\calJ)^{(n)}$ differ only in the first block entry, which implies $\Delta(\calJ^{(np)})^{(1)} = \Delta(\calJ)^{(n+1)}$. Thus
$$\mathfrak{m}_\Delta\big(\Delta(\calJ^{(np)})^{(1)}\big)=\mathfrak{m}_\Delta\big(\Delta(\calJ)^{(n+1)}\big)\to \mathfrak{m}_\Delta^0.
$$
Note that $\Delta(\calJ^0)^{(1)}$ is free, so $\mathfrak{m}_\Delta\big(\Delta(\calJ^0)^{(1)}\big)= \mathfrak{m}_\Delta^0$.
Therefore
$$\mathfrak{m}_\Delta\big(\Delta(\calJ^{(np)})^{(1)}\big)\to \mathfrak{m}_\Delta\big(\Delta(\calJ^0)^{(1)}\big).$$
Now use \eqref{p4_m-recur}: the first Jacobi parameters of $\Delta(\calJ^{(np)})$ converge to the first Jacobi parameters of $\Delta(\calJ^0)$, which implies that $\mathfrak{m}_\Delta\big(\Delta(\calJ^{(np)})\big)\to \mathfrak{m}_\Delta\big(\Delta(\calJ^0)\big)$ if $\widetilde{\Delta}(z)$ is a regular point of $\mathfrak{m}_\Delta\big(\Delta(\calJ^0)\big)$. This gives us $m(\calJ^{(np)})\to m(\calJ^0)$ by continuity of $g$, for all $z$ such that $\widetilde{\Delta}(z)$ is regular for $\mathfrak{m}_\Delta\big(\Delta(\calJ^0)\big)$.

In fact, note that the convergence $\mathfrak{m}_\Delta\big(\Delta(\calJ^{(np)})^{(1)}\big)\to \mathfrak{m}_\Delta\big(\Delta(\calJ^0)^{(1)}\big)$ is given by Lemma \ref{p4_convergence} to be uniform on compacts ($\mathfrak{m}_\Delta^0$ has no poles except at $\infty_-$). Therefore
\begin{multline*}
m(\calJ^{(np)})
=g\left(\left[B_{0}^{(np)}-\Delta(z)-A_{0}^{(np)} \mathfrak{m}_\Delta(\Delta(\calJ^{(np)})^{(1)})A_{0}^{(np)}{}^* \right]^{-1}, \right. \\
 \left.
 \vphantom{g\left(\left[B_{0}^{(np)}-\Delta(z)-A_{0}^{(np)} \mathfrak{m}_\Delta(\Delta(\calJ^{(np)})^{(1)})A_{0}^{(np)}{}^* \right]^{-1}, \right.}
 \{a_j\}_{j=1+np}^{N+np},\{b_j\}_{j=1+np}^{N+np}\right)
\end{multline*}
is just some rational function of finitely many uniformly convergent analytic functions. This implies that $m(\calJ^{(np)})$ is a sequence of meromorphic functions that converges to  $m(\calJ^0)$ uniformly on compacts {with respect to the spherical distance}. In particular, if $m(\calJ^0)$ has a pole at $z$, then $m(\calJ^{(np)})\to\infty$.
\qed \end{proof}

\end{subsection}

%%%%%%%%%%%%%%%%%%%%%%%%%%%%%%%%%%%%%%%%%%%%%%%%%%%%%%%%%%%%%%%%%%%%%%%%%%%%%%%%%%%%%%%%%%%%%

\begin{subsection}{Proof of Theorems \ref{p4_th_merom} and \ref{p4_th_finite_merom}} \label{Proofs_Proof}
\hspace*{\fill}
\smallskip
%\mbox{}\linebreak

\textit{Proof of Theorem \ref{p4_th_merom}.}

(ii)$\Rightarrow$(i) Passing from $m$ to $m^{(-2)}$ in Lemma \ref{p4_may4}, we may assume that $m$ itself satisfies (P1) and (P2).

We want to apply Lemma \ref{p4_lm_M} to $\Delta(\calJ)$.

(II)(A)  holds by (ii)(a) and  Lemma \ref{p4_may1}, and analytic continuation. Indeed, for any $\lambda\in\calR_R$, $\widetilde{f}_j(\lambda)\in\calS_R$, so all we need to check is continuity along $\pi^{-1}((-\infty,-2)\cup(2,\infty))\cap\calR_-$.
%Let $\eta\in (2,\infty)\setminus\{\Delta(\gamma_j)\}$, where $\gamma_j$ are the zeros of $\Delta'$, and consider
%\begin{equation}\label{p4_ee2.29}
%\lim_{\calR_-\cap\pi^{-1}(\bbC_+)\ni \lambda\to\eta_-} \mathfrak{m}_\Delta(\lambda) - \lim_{\calR_-\cap\pi^{-1}(\bbC_-)\ni \lambda\to\eta_-} \mathfrak{m}_\Delta(\lambda).
%\end{equation}
We want to show that for any $\eta\in (-\infty,-2)\cup(2,\infty)$,
\begin{equation}\label{p4_ee2.29}
\lim_{\calR_-\cap\pi^{-1}(\bbC_+)\ni \lambda\to\eta_-} \mathfrak{m}_\Delta(\lambda) = \lim_{\calR_-\cap\pi^{-1}(\bbC_-)\ni \lambda\to\eta_-} \mathfrak{m}_\Delta(\lambda).
\end{equation}
Even though in general $\lim_{\calR_-\cap\pi^{-1}(\bbC_+) \lambda\to\eta_-} \widetilde{f}_j(\lambda)\ne\lim_{\calR_-\cap\pi^{-1}(\bbC_-)\ni \lambda\to\eta_-} \widetilde{f}_j(\lambda) $, we however still have
\begin{equation*}\label{p4_ee2.30}
\left\{ \lim_{\calR_-\cap\pi^{-1}(\bbC_+) \lambda\to\eta_-} \widetilde{f}_j(\lambda) \right\}_{1\le j\le p} = \left\{ \lim_{\calR_-\cap\pi^{-1}(\bbC_-)\ni \lambda\to\eta_-} \widetilde{f}_j(\lambda) \right\}_{1\le j\le p} %= \left\{ (f_j(\eta))_- \right\}_{1\le j\le p}
\end{equation*}
as sets (these points just get permuted). Then \eqref{p4_ee2.6} shows that \eqref{p4_ee2.29} is true.
%One has to be careful with the points $\Delta(\gamma_j)_-$, since \eqref{p4_ee2.29} does not make sense at those. However one can deform the contour $\pi^{-1}((-\infty,-2)\cup(2,\infty))\cap\calR_-$ to avoid these points, and then repeat the above arguments. Thus (II)(A) holds.

(ii)(b), Lemma~\ref{p4_may1}, and Lemma~\ref{p4_may5} imply (II)(B).

Recall the functions $\mathfrak{U}(\lambda)$ and $\mathfrak{L}(\lambda)$, which we introduced in \eqref{funcU} and \eqref{funcL}. Equation \eqref{p4_ems} can be rewritten as
\begin{equation}\label{mmsharp}
  \left[\mathfrak{m}_\Delta(\lambda)-\mathfrak{m}^\sharp_\Delta(\lambda)\right]^{-1} = \mathfrak{U}(\lambda) \mathfrak{L}(\lambda)^{-1},
\end{equation}
and so the poles of \eqref{mmsharp} may come only from the poles of $\mathfrak{L}(\lambda)^{-1}$.

Let us show that (II)(C) holds. Assume that it does not, and there is a pole of \eqref{mmsharp} at $\lambda_0\in \pi^{-1}(F_R\setminus\{\pm2\})$. By symmetry, we can assume $\lambda_0\in\calR_+$.
%Without loss of generality, let $\lambda_0\in\calR_+\cap\pi^{-1}(\bbC_+\cup\bbR)$.

Suppose first that $\{ \widetilde{f}_j(\lambda_0) \}_{j=1}^p$ are all regular points for $m$ and $m^\sharp$.  Then $\mathfrak{L}(\lambda)$ is regular at $\lambda_0$. So the fact that $\mathfrak{L}(\lambda)^{-1}$ has a pole means that $\det\mathfrak{L}(\lambda)$ is zero. By Lemma \ref{p4_may6} and (ii)(b), $\lambda_0=\Delta(\gamma_j)_+$ for some $j$.
%
%Evaluating the determinant on the right-hand side of \eqref{p4_ems}, we see that $f_j(\lambda_0)=f_k(\lambda_0)$ for some $j,k$

By Lemma \ref{p4_may5}, $\mathfrak{L}(\lambda)^{-1}$ has a simple pole at $\lambda_0$%, where $k=\dim\spann \{v_1,\cdots,v_p\}^\perp$ and $F(\lambda_0)$ is regular. Then its Smith-McMillan form at $\lambda_0$ has $\{\lambda-\lambda_0,\ldots,\lambda-\lambda_0,1,\ldots,1\}$ ($k$ of $\lambda-\lambda_0$) on the diagonal. Therefore $F(\lambda)^{-1}$ has first order pole
, so
$$
\res_{\lambda=\lambda_0} \left[\mathfrak{m}_\Delta(\lambda)-\mathfrak{m}^\sharp_\Delta(\lambda)\right]^{-1} = \mathfrak{U}(\lambda_0) \res_{\lambda=\lambda_0} \mathfrak{L}(\lambda)^{-1}.
$$
But using Lemmas \ref{p4_lm0}, \ref{p4_may6}, and \ref{p4_may5}, we get
$$
\ran \res_{\lambda=\lambda_0} \mathfrak{L}(\lambda)^{-1}=\ker \mathfrak{L}(\lambda_0)=\ker \mathfrak{U}(\lambda_0),
$$
which implies $\res_{\lambda=\lambda_0} (\mathfrak{m}_\Delta(\lambda)-\mathfrak{m}^\sharp_\Delta(\lambda))^{-1}=\bdnot$, i.e., $(\mathfrak{m}_\Delta(\lambda)-\mathfrak{m}^\sharp_\Delta(\lambda))^{-1}$ is regular at $\lambda_0$.

Now assume that $z_0=\widetilde{f}_{n}(\lambda_0)$ is a pole for $m$ or $m^\sharp$ for some $1\le n\le p$. Note that \eqref{kerL} does not apply here, so we need some additional arguments.

By (ii)(d), $z_0$ cannot be a pole for both $m$ and $m^\sharp$. Without loss of generality, let it be a pole for $m^\sharp$. By the property (P2), $m(\widetilde{f}_j(\lambda_0))$ and $m^\sharp(\widetilde{f}_j(\lambda_0))$ are regular for $j\ne n$. By the property (P1), $\pi(\lambda_0) \ne \Delta(\gamma_j)$ for every $j$. Therefore $\mathfrak{U}(\lambda_0)$ is invertible. Let $k\ge1$ be the order of the pole of $m^\sharp$ at $z_0$. By Lemma~\ref{p4_may2}, $\mathfrak{m}_\Delta(\lambda)-\mathfrak{m}^\sharp_\Delta(\lambda)$ has a pole of order $k$ at $\lambda_0$. Let its Smith--McMillan  form (see Lemma~\ref{p4_lmSM}) be
\begin{equation*}
\mathfrak{m}_\Delta(\lambda)-\mathfrak{m}^\sharp_\Delta(\lambda)=E(\lambda) \diag\left((\lambda-\lambda_0)^{\kappa_1},\ldots,(\lambda-\lambda_0)^{\kappa_p}\right) F(\lambda)
\end{equation*}
with $\kappa_1\ge \kappa_2 \ge \ldots \ge\kappa_p=-k$. By Lemma \ref{p4_may3} (and (ii)(c)), $\det(\mathfrak{m}_\Delta(\lambda)-\mathfrak{m}^\sharp_\Delta(\lambda))$ has also a pole of order $k$. Therefore $\kappa_1+\ldots+\kappa_{p-1}=0$. In order to get that $(\mathfrak{m}_\Delta-\mathfrak{m}^\sharp_\Delta)^{-1}$ is regular at $\lambda_0$, we need to show that $\kappa_j\le 0$ for all $j$.

Using \eqref{p4_ems}, we can see that
$$
\lim_{\lambda\to\lambda_0} (\lambda-\lambda_0)^k [\mathfrak{m}_\Delta(\lambda)-\mathfrak{m}^\sharp_\Delta(\lambda)]
$$
has rank $1$, since each matrix $[p_{j-1}(f_j(\lambda))p_{s-1}(f_j(\lambda))]_{j,s=1}^p$ is of rank $1$ and $\mathfrak{U}(\lambda_0)$ is invertible. Therefore $\kappa_{p-1}>-k$.

Assume $0>\kappa_{p-1}>-k$. Then by Lemma \ref{p4_null-pole}, there exists an analytic $\bbC^p$-valued function $\phi_{p-1}$ such that $\phi_{p-1}(\lambda_0)\ne 0$ and
$$
(\lambda-\lambda_0)^{-\kappa_{p-1}} \phi_{p-1}(\lambda)^T (\mathfrak{m}_\Delta(\lambda)-\mathfrak{m}^\sharp_\Delta(\lambda))=\psi_{p-1}(\lambda)
$$
is analytic at $\lambda_0$ with $\psi_{p-1}(\lambda_0)\ne0$. Now plug $\mathfrak{m}_\Delta-\mathfrak{m}^\sharp_\Delta$ from \eqref{p4_ems} into the last expression.  We claim that, in fact,
$$
\lim_{\lambda\to \lambda_0}(\lambda-\lambda_0)^{-\kappa_{p-1}} \phi_{p-1}(\lambda)^T (\mathfrak{m}_\Delta(\lambda)-\mathfrak{m}^\sharp_\Delta(\lambda))=0.
$$
The reason is that $m(\widetilde{f}_{N}(\lambda_0))-m^\sharp(\widetilde{f}_{N}(\lambda_0))$ has a pole of order $k> -\kappa_{p-1}$, which forces $\phi_{p-1}(\lambda_0)^T$ to be in the kernel of $[p_{j-1}({f}_{N}(\lambda_0))p_{s-1}({f}_{N}(\lambda_0))]_{j,s=1}^p$. But any other $m(\widetilde{f}_j(\lambda_0))-m^\sharp(\widetilde{f}_j(\lambda_0))$ ($j\ne n$) is regular, so each of those terms in the sum vanishes too. Therefore we conclude $\psi_{p-1}(\lambda_0)=0$, a contradiction.

We showed that $\kappa_{p-1}\ge 0$. Since $\kappa_1+\ldots+\kappa_{p-1}=0$ and $\kappa_1\ge \kappa_2 \ge \ldots \ge\kappa_{p-1}\ge0$, we obtain $\kappa_1= \kappa_2 =\ldots =\kappa_{p-1}=0$, which implies that $(\mathfrak{m}_\Delta(\lambda_0)-\mathfrak{m}^\sharp_\Delta(\lambda_0))^{-1}$ is regular.

Finally we need to show that there are at most simple poles of $(\mathfrak{m}_\Delta-\mathfrak{m}^\sharp_\Delta)^{-1}$ at $\lambda_0=\pi^{-1}(\pm2)$. By (ii)(c), $m(\widetilde{f}_j(\lambda_0))-m^\sharp(\widetilde{f}_j(\lambda_0))$ are zeros of order at most 1. Let $k$ be the number of such simple zeros, and let the corresponding indices be $j_1,\ldots,j_k$. There are no poles of $m$ or $m^\sharp$ at $\widetilde{f}_j(\lambda_0)$ by (P1), so $m-m^\sharp$ is analytic there.
%Let the Smith--McMillan form of $(m_\Delta(\lambda)-m^\sharp_\Delta(\lambda))$ at $\lambda_0$ be
%\begin{equation}
%m_\Delta(\lambda)-m^\sharp_\Delta(\lambda)=E(\lambda) \diag\left((\lambda-\lambda_0)^{\kappa_1},\ldots,(z-z_0)^{\kappa_l}\right) F(\lambda)
%\end{equation}
%with $\kappa_1\ge \kappa_2 \ge \ldots \ge\kappa_l=-1$.
%Note that (using \eqref{p4_ems} again) $(\lambda-\lambda_0)(m_\Delta(\lambda_0)-m^\sharp_\Delta(\lambda_0))$ has rank $s$. This means that the last $s$ of $\kappa_j$'s are equal to $-1$: $\kappa_{l-s+1}=\ldots=\kappa_l=-1$. Finally note that any vector in $\spann \{v_{j_1},\cdots,v_{j_k}, v_{n_1}, \ldots, v_{n_s}\}^\perp$ is a
%
Repeating the arguments of Lemma \ref{p4_may6}, one sees that
$$\ker\mathfrak{L}(\lambda_0)=\spann\left[\{\vec{v}_{1},\cdots,\vec{v}_{p}\}\setminus\{\vec{v}_{j_1},\cdots,\vec{v}_{j_k}\}\right]^\perp,$$
where $\vec{v}_j=(1,p_1(f_j(\lambda_0)),\cdots,p_{p-1}(f_j(\lambda_0)))^*$. Since $\vec{v}_j$ are linearly independent, we see that the dimension of this kernel is precisely $k$. Since $\det\mathfrak{L}$ has a zero of order $k$ at $\lambda_0$  by Lemma \ref{p4_may6}, we conclude that its inverse has a simple pole (Lemma \ref{p4_lm0}). This establishes that $\mathfrak{m}_\Delta$ satisfies (II)(C).

Finally, let us check (II)(D). Assume $\mathfrak{m}_\Delta$ has a pole at $(\lambda_0)_+$ and $(\lambda_0)_-$ for some $\lambda\in\bbC\setminus{[-2,2]}$. By (P1), $\lambda_0\ne \Delta(\gamma_j)$ for any $j$. This implies that $\det\mathfrak{U}(\lambda_0)$ is invertible, and so the pole of $\mathfrak{m}_\Delta((\lambda_0)_+)$  must have come from  a pole of $m(f_j((\lambda_0)_+))$ or $m^\sharp(f_j((\lambda_0)_+))$ for some $j$. Similarly, the pole of $\mathfrak{m}_\Delta((\lambda_0)_-)$ comes from  a pole of $m(f_k((\lambda_0)_-))$ or $m^\sharp(f_k((\lambda_0)_-))$ for some $k$. But this violates the condition (P2).

Thus (II)(A)--(D) hold, and we are in position to apply Lemma  \ref{p4_lm_M}. Therefore $\Delta(\calJ)$ satisfies (I), which implies (i) by Lemma \ref{p4_DKS}.

\bigskip

(i)$\Rightarrow$(ii) %Applying Lemma \ref{p4_may4}, we might as well assume that $m$ satisfies ($\dagger$) and ($\ddagger$).
The condition (i) implies that (I) holds for $\Delta(\calJ)$ by Lemma \ref{p4_DKS}, which in turn implies that (II)(A)--(D) hold by Lemma \ref{p4_lm_M}.

(ii)(a) holds by Lemma \ref{p4_may0} and analytic continuation. Indeed, for each $l$, $1\le l \le p$, it allows us to meromorphically extend $m$ to  the region $\widetilde{f}_l(F_R)\cap\calS_-$. Their union is of course $\pi^{-1}(E_R)\cap\calS_-$, so the only thing we need to check is that our continuation is continuous on the boundaries of these regions, i.e., on $\pi^{-1}(\Delta^{-1}((-\infty,2)\cup(2,\infty)))\cap\calS_-$. Choose any $z_0$ there, and let $\lambda_0=\widetilde{\Delta}(z_0)$. Let us assume that $z_0$ lies on the boundaries of $\widetilde{f}_1(F_R)$ and of $\widetilde{f}_2(F_R)$. Then either
\begin{equation}\label{bdry}
\lim_{\calR_-\cap\pi^{-1}(\bbC_+)\ni \lambda\to\lambda_0}\widetilde{f}_1(\lambda)=\lim_{\calR_-\cap\pi^{-1}(\bbC_-)\ni \lambda\to\lambda_0}\widetilde{f}_2(\lambda)
\end{equation}
or
$$
\lim_{\calR_-\cap\pi^{-1}(\bbC_-)\ni \lambda\to\lambda_0}\widetilde{f}_1(\lambda)=\lim_{\calR_-\cap\pi^{-1}(\bbC_+)\ni \lambda\to\lambda_0}\widetilde{f}_2(\lambda).
$$
Without loss, let us assume it's \eqref{bdry}. We need to show
\begin{equation}\label{contin}
\lim_{\calS_-\cap \widetilde{f}_1(F_R) \ni z \to z_0} m(z) = \lim_{\calS_-\cap \widetilde{f}_2(F_R) \ni z \to z_0} m(z).
\end{equation}
But \eqref{bdry} implies
%Using \eqref{p4_ee2.30}, we  obtain
\begin{equation*}
\left\{ \lim_{\calR_-\cap\pi^{-1}(\bbC_+) \ni \lambda\to\lambda_0} \widetilde{f}_j(\lambda) \right\}_{1\le j\le p,j\ne1} = \left\{ \lim_{\calR_-\cap\pi^{-1}(\bbC_-)\ni \lambda\to\eta_-} \widetilde{f}_j(\lambda) \right\}_{1\le j\le p,j\ne2} %= \left\{ \widetilde{f}_j(\lambda_0) \right\}_{1\le j\le p}\setminus\{z_0\}.
\end{equation*}
Then \eqref{p4_func} and the fact that $\calJ-x_j$ commute for different $j$'s prove \eqref{contin}. Thus we established (ii)(a).

\eqref{p4_func} and (II)(B) imply (ii)(b).

Now let us show (ii)(c) and (ii)(d).

First of all, let $m^0$ be the Borel transform of the spectral measure of the periodic Jacobi matrix $\big(a_n^{(0)},b_n^{(0)}\big)_{n=1}^\infty$ from (i). Note that $m^0$ is of the form \eqref{m_periodic}, and it is straightforward to check that it satisfies (ii)(c) and (ii)(d) on all $\calS$ ((d) follows from the fact that $p_{p-1}(z)$ has simple zeros).

Note also that if $m(z)=m^\sharp(z)$ (this includes the possibility of $\infty=\infty$), then we would have $m^{(n)}(z)=m^{(n)}{}^\sharp(z)$ for every $n$ by \eqref{m_recur}. But by Lemma \ref{p4_convergence2} this would produce
$$
m^0(z)=\lim_{n\to\infty}m^{(n)}(z)=\lim_{n\to\infty}m^{(n)}{}^\sharp(z)=m^0{}^\sharp(z),
$$
which, as we just checked, is possible only if $z\in \pi^{-1}(\cup_{j=1}^p \{\alpha_j,\beta_j\})$.

Thus $m$ satisfies (ii)(c) and (ii)(d) with a possible exceptions of the band edges. So let us assume that $m(z)-m^\sharp(z)$ has a pole of order $k\ge2$ at some band edge $z_0\in \pi^{-1}(\cup_{j=1}^{p}\{\alpha_j,\beta_j\})$. Without loss of generality, we may assume $\widetilde{\Delta}(z_0)=2$ and $\widetilde{f}_1(2)=z_0$.

For $\lambda$ in a small neighborhood of $2$, let us define the $\bbC^p$-valued function $\phi(\lambda)$ to be the unique vector of norm $1$ in $\spann\{\vec{v}_2(\lambda),\vec{v}_3(\lambda),\ldots,\vec{v}_p(\lambda)\}^\perp$, where, just as in Lemma~\ref{p4_may6}, $\vec{v}_j(\lambda)=(1,p_1(f_j(\lambda)),\ldots,p_{p-1}(f_j(\lambda)))^*$. Indeed, for $\lambda$ close to $2$, this is a $1$-dimensional space. Moreover, $\phi(\lambda)$ is analytic at $\lambda=2$ (as a function on $\calS$), and $\phi(2)\ne 0$.

Now consider the function
$\phi(\lambda)^T \mathfrak{L}(\lambda)$ (see \eqref{funcL}). By construction, each term in the sum except $j=1$ is identically zero for any $\lambda$. The $j=1$ term has zero at $\lambda=2$ of order at least $k\ge2$ because of the factor $m(\widetilde{f}_1(\lambda))-m^\sharp(\widetilde{f}_1(\lambda))$. This means that $\phi(\lambda)$ is a left null function (see Definition \ref{null_pole}) at $\lambda=2$ for $\mathfrak{L}(\lambda)$ of order at least $2$. But that means that one of the $\kappa$'s in the Smith--McMillan form of $\mathfrak{L}(\lambda)$ is $\ge 2$. This implies that $\mathfrak{L}(\lambda)^{-1}$ has pole at $\lambda=2$ of order at least $2$. Then Lemma~\ref{p4_may2} implies that $(\mathfrak{m}_\Delta-\mathfrak{m}^\sharp_\Delta)^{-1}$ has a pole at $\lambda=2$ of order at least $2$, which contradicts (II)(C).
\qed

\medskip

\textit{Proof of Theorem \ref{p4_th_finite_merom}.}

(i)$\Rightarrow$(ii) If $\calJ=(a_n,b_n)_{n=1}^\infty$ is eventually periodic, then $\Delta(\calJ)$ is eventually free by the Magic Formula (Lemma~\ref{magic}). Then Lemma~\ref{p4_lm_M_finite} implies that $\mathfrak{m}_\Delta$ has a meromorphic continuation to the whole surface $\calR$. Lemma \ref{p4_may0} allows us to extend $m$ to the whole $\calS$ as well. Parts (ii)(b), (ii)(c), and (ii)(d) are already proven in the previous theorem.

\medskip

(ii)$\Rightarrow$(i) The result is obtained by following the proof of the previous theorem, but applying  Lemma \ref{p4_lm_M_finite}  instead of Lemma \ref{p4_lm_M} (note that $\mathfrak{m}_\Delta$ has meromorphic continuation to the whole surface $\calR$ by (ii)(a) and Lemma \ref{p4_may1}).
\qed
\end{subsection}
\end{section}

%%%%%%%%%%%%%%%%%%%%%%%%%%%%%%%%%%%%%%%%%%%%%%%%%%%%%%%%%%%%%%%%%%%%%%%%%%%%%%%%%%%%%%%%%%%%%%%%%
%%%%%%%%%%%%%%%%%%%%%%%%%%%%%%%%%%%%%%%%%%%%%%%%%%%%%%%%%%%%%%%%%%%%%%%%%%%%%%%%%%%%%%%%%%%%%%%%%
\begin{appendix}\label{Appendix}

%\begin{section}{}
\begin{section}{Orthogonal Polynomials on the Real Line}\label{Appendix_Orthogonal}
We will introduce some basics of orthogonal polynomials on the real line here. %Since we will also need the matrix-valued analogue of the theory, to avoid the repetition, we will do it simultaneously. Therefore we will immediately
We immediately start with the matrix-valued theory to avoid repetition. The scalar theory is of course a special case $p=1$. We will mention the differences between the scalar and matrix-valued cases as we proceed.

The proofs of most of the results listed here, along with more details, can be found in the paper by Damanik--Pushnitski--Simon \cite{DPS} (see also \cite{Rice}).

Let $\mu$ be a $p\times p$ matrix-valued Hermitian positive semi-definite finite measure on $\bbR$ of compact support, normalized
by $\mu(\bbR) = \bdone$, where $\bdone$ is the $p\times p$ identity matrix. For any $p\times p$ dimensional matrix functions $f,g$, define
\begin{equation*}
\lla f,g\rra_{L^2(\mu)}=\int f(x)^* d\mu(x) g(x);\label{p2_eq1.1}\\
%\lla f\rra^2_{L^2(\mu)}&=\lla f,f\rra_{L^2(\mu)},
\end{equation*}
where ${}^*$ is the Hermitian conjugation (just complex conjugation if $p=1$).
%Here we can regard $\lla f\rra_{L^2(\mu)}$ as the square root of the non-negative definite matrix $\lla f,f\rra_{L^2(\mu)}$. %By $\lla f,g\rra_{L^2}$, with the index just $L^2$, we will mean the product with respect to the Lebesgue measure on the real line or the unit circle, depending on the context.

What we have defined here is the right product of $f$ and $g$, as opposed to the left product $\int f(x)d\mu(x) g(x)^*$, whose properties are completely analogous.

Measure $\mu$ is called non-trivial if $|| \lla f,f\rra_{L^2(\mu)}||>0$ for all non-zero matrix-valued polynomials $f$. From now on assume $\mu$ is non-trivial. Then there exist unique (right) monic polynomials $\mathbf{P}^R_n$ of degree $n$ satisfying
\begin{equation*}
\lla \mathbf{P}^R_n,f \rra_{L^2(\mu)}=0 \quad \mbox{ for any polynomial } f \mbox{ with } \deg f<n.
\end{equation*}

For any choice of unitary $l\times l$ matrices $\tau_n$ (we demand $\tau_0=\bdone$), the polynomials
\begin{equation}\label{in2.3}
\mathfrak{p}^R_n=\mathbf{P}^R_n \lla \mathbf{P}^R_n,\mathbf{P}^R_n\rra_{L^2(\mu)}^{-1/2}\tau_n
\end{equation}
are orthonormal:
\begin{equation*}
\lla \mathfrak{p}^R_n,\mathfrak{p}^R_m \rra_{L^2(\mu)}=\delta_{n,m} \bdone,
\end{equation*}
where $\delta_{n,m}$ is the Kronecker $\delta$.
Using orthogonality one can show that they satisfy the (Jacobi) recurrence relation
\begin{equation}\label{p3_recur}
x \mathfrak{p}^R_n(x)=\mathfrak{p}^R_{n+1}(x)A_{n+1}^* +\mathfrak{p}^R_n(x)B_{n+1}+\mathfrak{p}^R_{n-1}(x) A_n, \quad n=1,2,\ldots,
\end{equation}
where matrices $A_n=\lla \mathfrak{p}^R_{n-1},x\mathfrak{p}^R_n \rra_{L^2(\mu)}$, $B_n=\lla \mathfrak{p}^R_{n-1},x\mathfrak{p}^R_{n-1} \rra_{L^2(\mu)}$ are called the Jacobi parameters (with $\mathfrak{p}^R_{-1}=\bdnot$, $A_0=\bdone$, the relation holds for $n=0$ too).

From the above recursion, it is easily seen that the leading coefficient of $\mathfrak{p}^R_n(x)$ is
\begin{equation}\label{leading}
(A_1^*)^{-1}\ldots (A_n^*)^{-1}.
\end{equation}

In the exact same fashion, just using the left product instead of right, one can define the left monic orthogonal polynomials $\mathbf{P}^L_n$ and left orthonormal polynomials $\mathfrak{p}^L_n$. It is not hard to see that $\mathbf{P}^L_n(z)=\mathbf{P}^R_n(\bar{z})^*$ and $\mathfrak{p}^L_n(z)=\mathfrak{p}^R_n(\bar{z})^*$.

We will be using the notation $\mathbf{P}_n$, $\mathfrak{p}_n$ for  matrix-valued polynomials, while in the case $p=1$ we will downgrade them to ${P}_n$, ${p}_n$. Whenever we write $\mathfrak{p}_n$ without the sup-index ${}^R$ or ${}^L$, we will mean the right orthonormal polynomial $\mathfrak{p}^R_n$.

Note that if $p=1$ it is natural to choose $\tau_n=1$ in \eqref{in2.3}. In particular this gives $p_n^R=p_n^L$, the Jacobi parameters become real, and $A_n$'s positive. This choice of $\tau_n$'s is not necessarily the best if $p>1$. Thus one has to talk about the equivalence classes of Jacobi matrices (see \cite{DPS,K_equiv}).

We can arrange sequences  $\{A_n\}_{n=1}^\infty$, $\{B_n\}_{n=1}^\infty$ (called Jacobi parameters) into an infinite matrix
\begin{equation}\label{in2.5}
\mathcal{J}=\left(
\begin{array}{cccc}
B_1&A_1&\mathbf{0}& \\
A_1^* &B_2&A_2&\ddots\\
\mathbf{0}&A_2^* &B_3&\ddots\\
 &\ddots&\ddots&\ddots\end{array}\right).
\end{equation}
This is called a block Jacobi matrix if $p>1$. If $p=1$ then we lose the word ``block'' and denote the Jacobi coefficients by $a_n$, $b_n$ instead of $A_n$, $B_n$.

If $A_n\equiv \bdone$, $B_n\equiv \bdnot$ the corresponding (block) Jacobi matrix is called \textit{free}.
%We will also use
%\begin{equation*}
%\calJ^{(k)}=\left(
%\begin{array}{cccc}
%B_{k+1}&A_{k+1}&\mathbf{0}&\cdots\\
%A_{k+1}^{*}&B_{k+2}&A_{k+2}&\cdots\\
%\mathbf{0}&A_{k+2}^*&B_{k+3}&\cdots\\
%\vdots&\vdots&\vdots&\ddots\end{array}\right), \quad
%\widetilde{\calJ}_k=\left(
%\begin{array}{cccccccc}
%B_1&A_1&\mathbf{0}\\
%A_1^{*}&B_2&A_2\\
%\mathbf{0}&A_2^*&\ddots&\ddots\\
%&&\ddots&\ddots&\ddots\\
%&&&A_{k-1}^* &B_k&A_k&\mathbf{0}\\
%&&&\mathbf{0}&A_k^*&\mathbf{0}&\mathbf{1}\\
%&&&\mathbf{0}&\mathbf{0}&\mathbf{1}&\mathbf{0}&\ddots\\
%&&&&&&\ddots&\ddots
%\end{array}\right)
%\end{equation*}

Conversely, any block Jacobi matrix \eqref{in2.5} with invertible $\{A_n\}_{n=1}^\infty$ gives rise to a $p\times p$ matrix-valued Hermitian measure $\mu$ via the spectral theorem. If $p=1$ this establishes a one-to-one correspondence between all non-trivial compactly supported measures and  bounded Jacobi matrices. If $p>1$ the same holds, except now the correspondence is with the set of equivalence classes of bounded block Jacobi matrices. This has the name of Favard's Theorem (see \cite{DPS} for a proof in the matrix-valued case).

%Since we will be considering perturbations of the free case in Sections 1.3.2--1.3.4, the following two classical results will prove to be useful.
%
%\begin{theorem}[Weyl's Theorem] If $A_n\to\bdone$, $B_n\to\bdnot$, then $\esssup \mu=[-2,2]$.
%\end{theorem}
%
%\begin{theorem}[Denisov--Rakhmanov Theorem]\label{denisov} Assume $\mu$ is a non-trivial $l\times l$ matrix-valued measure on $\bbR$ with associated block Jacobi matrix $\calJ$ of type $3$ such that $\esssup \mu=[-2,2]$ and $\det\left(\frac{d\mu(x)}{dx}\right)>0$ a.e. on $[-2,2]$. Then $A_n\to\bdone$, $B_n\to\bdnot$.
%\end{theorem}
%
%The first result is trivial, while the second, in the form given here, is proven in \cite{DKS} (see also \cite{yakhlef-marc}, as well as \cite{denisov, rakhmanov}).

Define the Borel transform (also called the Weyl-Titchmarsh ${m}$-function) of the measure $\mu$:
\begin{equation}\label{p3_m-func}
\mathfrak{m}(z)=\int \frac{d\mu(x)}{x-z},
\end{equation}
which is a matrix-valued meromorphic  function in $\bbC\setminus \esssup\mu$.
Again, we will use the letter $m$ instead of $\mathfrak{m}$ if $p=1$.

Define $\mathcal{J}^{(1)}$ to be the ``once-stripped'' Jacobi matrix with Jacobi parameters $(A_n,B_n)_{n=2}^\infty$, i.e., the Jacobi matrix of the form  \eqref{in2.5} with the first row and column removed. Then the following holds (the matrix-valued version is due to \cite{AN}):
\begin{equation}\label{m_recur}
A_1 \mathfrak{m}(z;\calJ^{(1)}) A_1^*= B_1-z-\mathfrak{m}(z;\calJ)^{-1}.
\end{equation}

As was explained in the Introduction, \cite{K_jost} established a connection between the rate of exponential convergence of Jacobi coefficients and meromorphic continuations of $\mathfrak{m}$.

Denote by $\calR=\calS_{[-2,2]}$ the Riemann surface corresponding to $[-2,2]$ (i.e., the hyperelliptic surface corresponding to the polynomial $z^2-4$). Recall Definitions~\ref{s} and~\ref{sharp}.

Let $x(z)=z+z^{-1}$, and for any $R>1$ let $\calR_R=\calR_+ \cup \pi^{-1}(F_R)$, where $F_R$ is the interior of the bounded component of $x(R\,\partial\bbD)$ (ellipse).

%Now assume that $\fre=[-2,2]$. Instead of looking at the meromorphic continuations of $m$ through $(-2,2)$, it will be convenient to move $\bbC\setminus [-2,2]$ to $\bbD$ via the inverse of $z \mapsto z+z^{-1}$, and study the meromorphic continuations of $$M(z)=-m(z+z^{-1})$$ from $\bbD$ through $\partial\bbD$. Note that $M$ is also Herglotz in the meaning that $\imag M(z)\gtrless \bdnot$ if $z\in\bbC_{\pm}\cap \bbD$.
%
%Let $M^\sharp(z)=M(\overline{z}^{-1})^\dagger$ (which corresponds to $m^\sharp(z_\pm)=m(\overline{z}_\mp)^\dagger$).
%The following result hold.

The next three lemmas are taken from the author's~\cite{K_jost}.

\begin{lemma}\label{p4_lm_M}
Let $\esssup \mu = [-2,2]$ and $R>1$. Define $\mathfrak{m}$ as in \eqref{p3_m-func}. The following are equivalent:
\begin{list}{}
{
\setlength\labelwidth{2em}%
\setlength\labelsep{0.5em}%
\setlength\leftmargin{2.5em}%
}
\item[(I)] The Jacobi matrix $(A_n,B_n)_{n=1}^\infty$ associated with $\mu$ satisfies
  \begin{equation*}
  \limsup_{n\to\infty}\left(||B_n||+||\mathbf{1}-A_nA_n^*||\right)^{1/2n}\le R^{-1}.
  \end{equation*}
\item[(II)] All of the following holds:
\begin{itemize}
\item[(A)] $\mathfrak{m}$ has a meromorphic continuation to $\calR_R$;
\item[(B)] $\mathfrak{m}$ has no poles in $\pi^{-1}(-2,2)$, and at most simple poles at $\pi^{-1}(\pm2)$;
\item[(C)] $(\mathfrak{m}-\mathfrak{m}^\sharp)^{-1}$ has no poles in $\pi^{-1}(F_R)$, except at $\pi^{-1}(\pm2)$, where they are at most simple;
%\item For each $R^{-1}<|z_j|<1$, if $M(z)$ has at most order $1$ pole at $z_j^{-1}$ then
%\begin{gather}
%\label{p4_eq3.55n}
%\ran \widetilde{w}_j \subset \ran (\widetilde{w}_j+z_j^2 \widetilde{q}_j),\\
%\label{p4_eq3.56n}
%\ran \widetilde{w}_j \cap \ran \widetilde{q}_j=\varnothing,
%\end{gather}
%where $\widetilde{q}_j=\lim_{z\to z^{-1}_j} (z-z^{-1}_j)M(z)$.
\item[(D)] If $\mathfrak{m}$ has a pole at $\lambda_0\in \pi^{-1}(F_R) \cup \calR_+$ and at $\lambda_0^\sharp$, then%if $M(z)$ has a pole at $z_j^{-1}$ of order higher than $1$, then $(M(z)-M^\sharp(z))^{-1} M(z)$  is analytic at $z_j^{-1}$ and
\begin{gather*}
\ran \res_{\lambda=\lambda_0}\mathfrak{m}(\lambda) \subset \ker (\mathfrak{m}(\lambda_0^\sharp)-\mathfrak{m}^\sharp(\lambda_0^\sharp))^{-1} ,\\
\ran \res_{\lambda=\lambda_0}\mathfrak{m}(\lambda) \subset \left(\ran (\mathfrak{m}(\lambda_0^\sharp)-\mathfrak{m}^\sharp(\lambda_0^\sharp))^{-1} \mathfrak{m}(\lambda_0^\sharp)\right)^\perp.
%=\ker M(z^{-1}_j) (M(z^{-1}_j)-M^\sharp(z^{-1}_j))^{-1} .
\end{gather*}
\end{itemize}
\end{list}
\end{lemma}
%\begin{remark}
%Let $z_j$ be a pole of $M$ in $\{z: R^{-1}<|z|<1\}$. If $M$ has no pole at $z_j^{-1}$, then the condition (c) is vacuous.
%\end{remark}

\begin{lemma}\label{p4_lm_M_finite}
Let $\esssup \mu = [-2,2]$. Define $\mathfrak{m}$ as in \eqref{p3_m-func}. The following are equivalent:
\begin{list}{}
{
\setlength\labelwidth{2em}%
\setlength\labelsep{0.5em}%
\setlength\leftmargin{2.5em}%
}
\item[(I)] The Jacobi matrix $(A_n,B_n)_{n=1}^\infty$ associated with $\mu$ satisfies
\begin{equation*}||B_n||+||\mathbf{1}-A_nA_n^*||=\bdnot \quad \mbox{for all large }n.\end{equation*}
\item[(II)] All of the following holds:
\begin{itemize}
\item[(A)] $\mathfrak{m}$ is a rational matrix function;
\item[(B)] $\mathfrak{m}$ has no poles in $\pi^{-1}(-2,2)$, and at most simple poles at $\pi^{-1}(\pm2)$;
\item[(C)] $(\mathfrak{m}-\mathfrak{m}^\sharp)^{-1}$ has no poles in $\calR_R$, except at $\pi^{-1}(\pm2)$, where they are at most simple;
\item[(D)] If $\mathfrak{m}$ has a pole at $\lambda_0\in\calR_+$, and at $\lambda_0^\sharp$, then
\begin{gather*}
\ran \res_{\lambda=\lambda_0}\mathfrak{m}(\lambda) \subset \ker (\mathfrak{m}(\lambda_0^\sharp)-\mathfrak{m}^\sharp(\lambda_0^\sharp))^{-1} ,\\
\ran \res_{\lambda=\lambda_0}\mathfrak{m}(\lambda) \subset \left(\ran (\mathfrak{m}(\lambda_0^\sharp)-\mathfrak{m}^\sharp(\lambda_0^\sharp))^{-1} \mathfrak{m}(\lambda_0^\sharp)\right)^\perp.
\end{gather*}
\end{itemize}
\end{list}
\end{lemma}

\begin{remarks}
1. We stated these lemmas in terms of $\mathfrak{m}$, rather than of $M$ (see \eqref{capM}) as it was in \cite{K_jost}. Note also that $z_0^{-1}$ in \cite[Thm~3.8(D)/3.9(D)]{K_jost} should be better thought of as $\bar{z}_0^{-1}$. Since the only poles of $\mathfrak{m}$ on $\calR_+$ are the pure points of the spectral measure, they must be real, and so $z_0^{-1}=\bar{z}_0^{-1}$.

%that the only poles of $\mathfrak{m}$ on $\calR_+$ are pure points of the spectral measure, and so must be real. So the only $\lambda$'s that we have to worry about in (D) are in $\pi^{-1}(\bbR)$. This explains why $z_0^{-1}$ from \cite{K_jost} is the same as $\bar{z}_0^{-1}$.

2. Conditions (D) of Lemmas~\ref{p4_lm_M}/\ref{p4_lm_M_finite} do not look pleasant. Note however that they  are trivially satisfied if no $\lambda$ and $\lambda^\sharp$ are both poles of $\mathfrak{m}$. This will be enough for our purposes. One can also show that for $p=1$, (D) is equivalent to $\mathfrak{m}$ not having simultaneous poles at $\lambda$ and $\lambda^\sharp$ (see \cite{K_jost}).
\end{remarks}

\begin{lemma}\label{p4_convergence}
Under the conditions of one of the previous two lemmas,
\begin{equation*}
\mathfrak{m}^{(n)}(\lambda) \to \mathfrak{m}^0(\lambda) \quad \mbox{ uniformly on compacts of} \quad \calR_R,
%\frac{-\lambda+\sqrt{\lambda^2-4}}{2} \bdone \quad \mbox{ uniformly on compacts of} \quad \calR_R,
\end{equation*}
where $\mathfrak{m}^{(n)}$ is the Borel transform of the spectral measure for the $n$ times stripped operator $\calJ^{(n)}$, and $\mathfrak{m}^0$ is the Borel transform of the spectral measure for the free block Jacobi matrix.
\end{lemma}
\begin{remarks}
1. In fact, $\mathfrak{m}^0(\lambda)=\frac{-\lambda\pm \sqrt{\lambda^2-4}}{2} \bdone$.

2. Convergence on compacts of $\calR_+$ (but not $\calR_R$) is obvious from the convergence of the resolvents.
\end{remarks}

%%%%%%%%%%%%%%%%%%%%%%%%%%%%%%%%%%%%%%%%%%%%%%%%%%%%%%%%%%%%%%%%%%%%%%%%%%

Let us define the second kind polynomials by
\begin{equation*}
\mathfrak{q}_n^R(z)=\int_\bbR d\mu(x) \frac{\mathfrak{p}_n^R(z)-\mathfrak{p}_n^R(x)}{z-x}, \quad n=0,1,\ldots .
\end{equation*}
It can be shown that $\mathfrak{q}_n^R$ are polynomials of degree $n-1$, and that they satisfy the same recurrence relations \eqref{p3_recur}.  For future reference,
\begin{align} \label{p4_ee1.11}
\mathfrak{p}_0^R(z) &= \bdone, \quad \mathfrak{p}_1^R(z)=(z-B_1) A_1^*{}^{-1},
\\
 \label{p4_ee1.12}
\mathfrak{q}_0^R(z)&=\bdnot, \quad \mathfrak{q}_1^R(z)= A_1^*{}^{-1}.
\end{align}
Define also $\mathfrak{q}_n^L=\mathfrak{q}_n^R(\bar{z})^*$.

The resolvent of $\calJ$ has the following block form (see \cite[Thm 2.29]{DPS})
\begin{equation}\label{p4_resolvent}
(\calJ-z)^{-1}=\left(
\begin{array}{cccc}
\mathfrak{m} & \mathfrak{q}_1^R+\mathfrak{m} \mathfrak{p}_1^R & \mathfrak{q}_2^R+\mathfrak{m} \mathfrak{p}_2^R & \cdots \\
\mathfrak{q}_1^L+\mathfrak{p}_1^L \mathfrak{m} & \mathfrak{q}_1^L \mathfrak{p}_1^R+\mathfrak{p}_1^L \mathfrak{m} \mathfrak{p}_1^R & \mathfrak{p}_1^L \mathfrak{q}_2^R+\mathfrak{p}_1^L \mathfrak{m} \mathfrak{p}_2^R &\cdots \\
\mathfrak{q}_2^L+\mathfrak{p}_2^L \mathfrak{m} & \mathfrak{q}_2^L \mathfrak{p}_1^R+\mathfrak{p}_2^L \mathfrak{m} \mathfrak{p}_1^R & \mathfrak{q}_2^L \mathfrak{p}_2^R+\mathfrak{p}_2^L \mathfrak{m} \mathfrak{p}_2^R &\cdots \\
\vdots &\vdots &\vdots&\ddots
\end{array}
\right)(z),
\end{equation}
i.e., its $(i,j)$-th block entry is $\mathfrak{q}_{i-1}^L \mathfrak{p}_{j-1}^R+\mathfrak{p}_{i-1}^L \mathfrak{m} \mathfrak{p}_{j-1}^R$  if $i\ge j$, and $\mathfrak{p}_{i-1}^L \mathfrak{q}_{j-1}^R+\mathfrak{p}_{i-1}^L \mathfrak{m} \mathfrak{p}_{j-1}^R$ otherwise.

\end{section}

\medskip
%%%%%%%%%%%%%%%%%%%%%%%%%%%%%%%%%%%%%%%%%%%%%%%%%%%%%%%%%%%%%%%%%%%%%%%

\begin{section}{Periodic Jacobi Matrices}\label{Appendix_Periodic}

By periodic Jacobi matrices we mean the (scalar) Jacobi matrices satisfying \eqref{p4_periodic} for some $p$.
%By periodic orthogonal polynomials we mean the orthogonal polynomials corresponding to such Jacobi matrices.
We already mentioned some properties of them in Section \ref{Prelim_Periodic}. In particular we introduced the notion of discriminant $\Delta$ of a periodic Jacobi matrix. As we already mentioned, $\Delta$ determines the essential spectrum of $\calJ$. Next lemma contains some further properties of $\calJ$ and $\Delta$. %For a proof, see, e.g., \cite{Rice}.

\begin{lemma}\label{discriminant}
Let $\calJ$ be a $($one-sided$)$ $p$-periodic Jacobi matrix, $m$ the Borel transform of the spectral measure, and $\Delta$ its discriminant \eqref{p4_discr}. Then
\begin{list}{}
{
\setlength\labelwidth{2em}%
\setlength\labelsep{0.5em}%
\setlength\leftmargin{2.5em}%
}
\item[(i)]
\begin{itemize}
\item[$\bullet$] $\Delta^{-1}([-2,2])\subset\bbR$.
\item[$\bullet$] Let $x_1^\pm\le x_2^\pm \le\ldots \le x_p^\pm$ be the zeros $($counting multiplicity$)$ of $\Delta(\lambda)\mp 2$. Then
\begin{equation*}
x_p^+>x_p^-\ge x_{p-1}^->x_{p-1}^+ \ge x_{p-2}^+>x_{p-2}^-\ge\ldots.
\end{equation*}
\item[$\bullet$] $\Delta(\lambda)$ is strictly increasing on each interval $(x_{p-2j}^-,x_{p-2j}^+)$ and strictly decreasing on each interval $(x_{p-2j-1}^+,x_{p-2j-2}^-)$, $j=0,1,\ldots$. In particular the $p-1$ solutions  of $\Delta'(\lambda)=0$ are all real and lie one per each gap. If a gap is open, then the corresponding solution lies in the gap's interior.
\end{itemize}
\item[(ii)] $m$ has a meromorphic continuation to $\calS_\fre$ and its two branches are given by
\begin{equation}\label{m_periodic}
m(z)=\frac{-\beta(z)\pm \sqrt{\beta(z)^2-4\alpha(z)\gamma(z)}}{2\alpha(z)},
\end{equation}
where $\alpha(z)=a_p p_{p-1}(z)$, $\beta(z)=p_p(z)+a_p q_{p-1}(z)$, $\gamma(z)=q_p(z)$. Moreover,
$$
\beta(z)^2-4\alpha(z)\gamma(z) = \Delta(z)^2-4.
$$
%\item[(i)] $\Delta(z)=p_p(z)-a_p q_{p-1}(z)$, where $p_j,q_j$ are orthogonal polynomials of the first and second kind of $\calJ$.
%\item[(ii)] $\Delta(z)=\frac{1}{a_1\ldots a_p} \prod_{j=1}^p (z-b_j)+O(z^{p-2})$.
%\item[(iii)]
\end{list}
\end{lemma}

There is a nice connection between the theory of periodic orthogonal polynomials and matrix-valued orthogonal polynomials.
Note that applying a polynomial of degree $p$ to the tridiagonal matrix $\calJ$ gives us $(2p+1)$-diagonal matrix, which can be viewed as a block Jacobi matrix with $p\times p$ matrix-valued Jacobi parameters $A_n, B_n$ (note that $A_n$ are lower triangular).

Let $S$ be the right shift operator on $\ell^2(\bbZ)$. Note that $S^p+S^{-p}$ is the free block Jacobi matrix with $p\times p$ block entries. % a block  Jacobi matrix with $p\times p$ blocks $\bdnot$ on diagonal and $\bdone$ above and below it (the ``free'' block Jacobi matrix).

\begin{lemma}[``Magic Formula'', Damanik--Killip--Simon \cite{DKS}]\label{magic}
Let $\calJ_0$ be a $p$-periodic Jacobi matrix with discriminant $\Delta_{\calJ_0}$ and isospectral torus $\calT_\fre$. Let $\calJ$ be any two-sided Jacobi matrix. Then
\begin{equation*}
\Delta_{\calJ_0}(\calJ)=S^p+S^{-p} \quad \Leftrightarrow \quad \calJ\in \calT_\fre.
\end{equation*}
\end{lemma}

Moreover we can ``perturb'' this result if all gaps are open.

\begin{lemma}\label{p4_DKS}
Let $\calJ_0$ be a $p$-periodic Jacobi matrix with discriminant $\Delta_{\calJ_0}$ and isospectral torus $\calT_\fre$, such that all gaps of $\calJ_0$ are open $($every interval of $\mathfrak{e}$ has equal equilibrium measure$)$. Let $\calJ$ be any two-sided Jacobi matrix, and let $(A_n,B_n)_{n\in\bbZ}$ be the $p\times p$ Jacobi parameters of $\Delta_{\calJ_0}(\calJ)$. Then the following are equivalent:
\begin{list}{}
{
\setlength\labelwidth{2em}%
\setlength\labelsep{0.5em}%
\setlength\leftmargin{2.5em}%
}
\item[(i)] $  \limsup_{n\to\infty} \left(|a_n-a_n^{(0)}| + |b_n-b_n^{(0)}|\right)^{1/2n}  \le R^{-1},$ where $\big(a_n^{(0)},b_n^{(0)}\big)_{n=1}^\infty$ is a periodic Jacobi matrix from $\calT_\fre$.
\item[(I)] $\limsup_{n\to\infty} (||\bdone-A_nA_n^*||+||B_n||)^{1/2n} \le R^{-1}$.
\end{list}
\end{lemma}
\begin{remark}
Since both conditions depend on the behavior of the coefficients at $+\infty$, this result can also be applied to one-sided Jacobi matrices $\calJ$.
\end{remark}
\begin{proof}
The proof of this lemma requires some slight modifications of the arguments of Damanik--Killip--Simon \cite{DKS}.

First of all, notice that there exist positive constants $c_1, c_2$ such that for all $A$ in a neighborhood of $\bdone$,
\begin{equation*}
c_1 ||\bdone-|A| \, || \le ||\bdone - AA^* || \le c_2 ||\bdone - |A| \, ||.
\end{equation*}
Moreover, for the lower-triangular matrices $A$ in a neighborhood of $\bdone$, we can find positive constants $\widetilde{c}_1, \widetilde{c}_2$ so that
\begin{equation}\label{matrix_in}
\widetilde{c}_1 ||\bdone-A \, || \le ||\bdone - AA^* || \le \widetilde{c}_2 ||\bdone - A \, ||
\end{equation}
(see \cite[Prop 11.12]{DKS}).

(i)$\Rightarrow$(I)  As was shown in \cite[Section 11]{DKS}, each $\bdone-A_n$ and $B_n$ is a smooth function of $p$ consecutive pairs $(a_n,b_n)$. Moreover, each one of them vanish on $\calT_\fre$. Therefore (i), Lipschitz property of smooth functions, and \eqref{matrix_in} imply (I).

\smallskip

(I)$\Rightarrow$(i) Note that $\lim_{n\to\infty} A_nA_n^* =\bdone$ implies $\lim_{n\to\infty} A_n = \bdone$ because of the fact that $A_n$'s are lower triangular (see \cite[Thm 2.9]{DPS}, as well as \cite{K_equiv}). Therefore $A_n$'s are in a neighborhood of $\bdone$, and so they eventually all satisfy \eqref{matrix_in}. Therefore we may assume that
$$
\limsup_{n\to\infty} ||\bdone - A_n ||_{HS}^{1/2n} \le R^{-1}, \quad \limsup_{n\to\infty} ||B_n||_{HS}^{1/2n} \le R^{-1},
$$
where $||\cdot||_{HS}$ is the Hilbert--Schmidt norm. Then following the same arguments as \cite[Lemma~11.11]{DKS} and then \cite[Theorem~11.13(i)$\Rightarrow$(vi)]{DKS}, we obtain that
\begin{equation}\label{d_n_minim}
\limsup_{n\to\infty} d_n(\calJ,\calT_\fre)^{1/2n} \le R^{-1},
\end{equation}
where $d_n(\calJ,\calT_\fre) = \inf \{ d_n(\calJ,\calJ_0): \calJ_0\in\calT_\fre\}$, where
\begin{equation*}
d_n((a_j,b_j),(a_j',b_j'))=\sum_{j=0}^\infty e^{-j} (|a_{n+j}-a'_{n+j}|+|b_{n+j}-b'_{n+j}|).
\end{equation*}
But this implies that there exists some $\calJ_0\in\calT_\fre$ so that (i) holds. Indeed, denote $\big(a_j^{(n)},b_j^{(n)}\big)_{j=1}^p$ to be the periodic Jacobi matrix from $\calT_\fre$ that minimizes $d_n(\calJ,\calT_\fre)$. Then \eqref{d_n_minim} implies
$$
\sup_{1\le j \le p} |a_j^{(n)}-a_j^{(n+1)}| + \sup_{1\le j \le p} |b_j^{(n)}-b_j^{(n+1)}| \le C R^{-2n}.
$$
This means that for each $j=1,\ldots,p$, the sequences $a_j^{(n)}$ and $b_j^{(n)}$ are Cauchy with respect to $n\to\infty$, and therefore have limits $a_j^{(0)}$ and $b_j^{(0)}$ respectively, to which they converge exponentially fast. If we let $\calJ_0=\big(a_j^{(0)},b_j^{(0)}\big)_{j=1}^p$, then \eqref{d_n_minim} gives $\limsup_{n\to\infty} d_n(\calJ,\calJ_0)^{1/2n} \le R^{-1}$, which implies (i).
\qed \end{proof}

\end{section}

%%%%%%%%%%%%%%%%%%%%%%%%%%%%%%%%%%%%%%%%%%%%%%%%%%%%%%%%%%%%%%%%%%%%%%%%%%%%%%%%%%%%%%%%%%%%%%%%%%%%

\begin{section}{Matrix-Valued Functions}\label{Appendix_Matrixvalued}

Throughout the paper, all meromorphic/analytic matrix functions are assumed to have  not identically vanishing  determinant.

The order of a pole of a $p\times p$ matrix-valued meromorphic function $f$ is defined to be the minimal $k>0$ such that $\lim_{z\to z_0} (z-z_0)^k f(z)$ is a finite non-zero matrix.

By a zero of a matrix-valued meromorphic function $f$ we call a point at which $f^{-1}$ has a pole.

Denote by $\delta_j\in\bbC^p$, $1\le j\le p$, the column vector having $1$ on the $j$-th position, and $0$ everywhere else.

We will make use of the so-called (local) Smith--McMillan form (see, e.g.,~\cite[Thm 3.1.1]{Ball}).

\begin{lemma}\label{p4_lmSM}
Let $f(z)$ be a $p\times p$ matrix-valued function meromorphic at $z_0$ with determinant not identically zero. Then $f(z)$ admits the representation
\begin{equation}\label{p4_ee1.31}
f(z)=E(z) \diag\left((z-z_0)^{\kappa_1},\ldots,(z-z_0)^{\kappa_p}\right) F(z),
\end{equation}
where $E(z)$ and $F(z)$ are $p\times p$ matrix-valued functions which are analytic and invertible in a neighborhood of $z_0$, and $\kappa_1\ge \kappa_2 \ge \ldots \ge\kappa_p$ are integers (positive, negative, or zero).
\end{lemma}

This immediately gives us the following corollary.

\begin{lemma}\label{p4_lm0}
%\begin{itemize}
%\item[(i)]
Let $u$ be an analytic function at $z_0$ such that $z_0$ is a zero of $\det u$ of order $k>0$. Then $\dim \ker u(z_0)=k$ if and only if $z_0$ is a pole of $u(z)^{-1}$ of order $1$. % (meaning that $\res_{z=z_0} u(z)^{-1} \equiv\lim_{z\to z_0} (z-z_0) u(z)^{-1}$ is a finite non-zero matrix).

If this is the case, then
\begin{align*}
\ker\res_{z=z_0} u(z)^{-1} &=\ran u(z_0), \\
\ran \res_{z=z_0} u(z)^{-1} &=\ker u(z_0).
\end{align*}
\end{lemma}
%\item[(ii)] Let
%\end{itemize}
\begin{proof}
Both of the conditions in the if-and-only-if statement are equivalent to saying that $\kappa_1=\ldots=\kappa_k=1$, $\kappa_{k+1}=\ldots=\kappa_p=0$ in the Smith-McMillan form of $u(z)$ at $z_0$. Then note that both $\ker\res_{z=z_0} u(z)^{-1}$ and $\ran u(z_0)$ are equal to $E(z_0)\spann\left\{\delta_{k+1},\cdots,\delta_p\right\}$. Similarly, both $\ran \res_{z=z_0} u(z)^{-1} $ and $\ker u(z_0)$ are equal to $F(z_0)^{-1} \spann\left\{\delta_{1},\cdots,\delta_k\right\}$.
\qed \end{proof}

\begin{definition} \label{null_pole}
\hspace*{\fill}
\begin{list}{}
{
\setlength\labelwidth{2em}%
\setlength\labelsep{0.5em}%
\setlength\leftmargin{2.5em}%
}
\item[(i)] An analytic $\bbC^p$-valued function $\phi(z)$ with $\phi(z_0)\ne0$ is called a {left }%(right)
{null function} at $z_0$ of order $k>0$ for a meromorphic matrix-valued function $f$, if $\phi(z)^T f(z)$ %($f(z)\phi(z)$ resp.)
is analytic at $z_0$ with a zero of order $k$ at $z_0$.
\item[(ii)] An analytic $\bbC^p$-valued function $\psi(z)$ with $\psi(z_0)\ne0$ is called a {left }%(right)
{pole function } at $z_0$ of order $k>0$ for a meromorphic matrix-valued function $f$, if there exists an analytic $\bbC^p$-valued function $\phi(z)$ with $\phi(z_0)\ne0$ such that $\phi(z)^T f(z)=(z-z_0)^{-k}\psi(z)$. %($f(z)\phi(z)=(z-z_0)^{-k}\psi(z)$ resp.).
\end{list}
\end{definition}
Note that $\psi$ is a left %(right)
pole function for $f$ if and only if $\psi$ is a left %(right)
null function for $f^{-1}$.

The following is immediate from the definition and will prove to be useful for us. %An easy observation is that if $f$ has a local Smith-McMillan form as in Lemma \ref{p4_lmSM}, then
\begin{lemma}\label{p4_null-pole}
Let $f$ has a local Smith--McMillan form \eqref{p4_ee1.31} with $\kappa_1\ge\ldots\ge \kappa_j>0$, $0>\kappa_{r}\ge\ldots\ge \kappa_p$. Then
\begin{list}{}
{
\setlength\labelwidth{2em}%
\setlength\labelsep{0.5em}%
\setlength\leftmargin{2.5em}%
}
\item[(i)] Functions $(E(z)^{-1})^T \delta_1, \ldots ,(E(z)^{-1})^T \delta_j$ are left null functions for $f(z)$ at $z_0$ of orders $\kappa_1,\ldots,\kappa_j$, respectively.
%
%    $F(z)^{-1} \delta_1, \ldots F(z)^{-1} \delta_j$ are right null functions for $f$ at $z_0$ of order $\kappa_1,\ldots,\kappa_j$ respectively.
\item[(ii)]  Functions  $\delta^T_{r}F(z) , \ldots ,\delta^T_{p}F(z)$ are left pole functions for $f(z)$ at $z_0$ of orders $-\kappa_{r},\ldots,-\kappa_p$, respectively.
%
%    $E(z) \delta_{l-r+1}, \ldots E(z) \delta_{l}$ are right pole functions for $f$ at $z_0$ of order $-\kappa_{l-r+1},\ldots,-\kappa_l$ respectively.
\end{list}
\end{lemma}
\end{section}

%%%%%%%%%%%%%%%%%%%%%%%%%%%%%%%%%%%%%%%%%%%%%%%%%%%%%%%%%%%%%%%%%%%%%%%%%%%%%

\begin{section}{Herglotz functions}\label{Appendix_Herglotz}

\begin{definition} An analytic in $\bbC_+$ $l\times l$ matrix-valued function $\mathfrak{m}$ is called {Herglotz} if $\imag \mathfrak{m}(z)\ge \bdnot$ for all $z\in  \bbC_+$.
\end{definition}
Here $\imag T\equiv \frac{T-T^*}{2i}$.

We can also define $\mathfrak{m}$ on the lower half plane $\bbC_-$ by reflection $\mathfrak{m}(z)=\mathfrak{m}(\bar{z})^*$, so that $\imag \mathfrak{m}(z) \le \bdnot$ for all $z$ with $\imag z<0$. In particular the function $\mathfrak{m}$ defined in \eqref{p3_m-func} is Herglotz.

We will assume from now on that $\det\imag \mathfrak{m}(z)$ is not identically zero, in which case  the inequality in $\imag \mathfrak{m}(z)\gtrless \bdnot$ is everywhere strict (see \cite[Lemma 5.3]{Ges}).

The following result is well-known (see, e.g., \cite[Thm 5.4]{Ges}).

\begin{lemma}\label{p4_ges}
Let $\mathfrak{m}$ be an $l\times l$ matrix-valued Herglotz function. Then there exist an $l \times l$ matrix-valued measure $\mu$ on $\bbR$ satisfying $\int_\bbR \frac{1}{1+x^2}d\mu(x)<\infty$, and constant matrices $C=C^*, D\ge\bdnot$ such that
\begin{equation}\label{p4_herg}
\mathfrak{m}(z)=C+Dz+\int_\bbR \left(\frac{1}{x-z}-\frac{x}{1+x^2}\right) d\mu(x), \quad z\in \bbC_+.
\end{equation}
The absolutely continuous part of $\mu$ can be recovered from this representation by
\begin{equation*}\label{p2_herg1}
f(x)\equiv\frac{d\mu}{dx}=\pi^{-1}\lim_{\veps\downarrow0} \imag \mathfrak{m}(x+i \veps),
\end{equation*}
and the pure point part by
\begin{equation*}\label{p2_herg2}
\mu(\{\lambda\})=\lim_{\veps \downarrow 0} \veps\, \imag \mathfrak{m}(\lambda+i\veps) = \lim_{\veps\downarrow 0} \veps\, \mathfrak{m}(\lambda+i\veps).
\end{equation*}
\end{lemma}

%\begin{definition}\label{disctrete}
%A \textbf{discrete $m$-function} is a Herglotz function, $m(z)$, which has an analytic continuation from $\bbC_+$ to $\bbC\setminus I$ for some bounded interval $I\subset \bbR$, and satisfies
%\begin{align*}
%&z\in\bbR\setminus I \Rightarrow \imag m(z)=\bdnot, \\
%&m(z)=z^{-1}\bdone +O(z^{-2}) \mbox{ at }\infty.
%\end{align*}
%\end{definition}
%
%The following is immediate from Lemma \ref{p4_ges}.
%\begin{lemma}
%A function $m(z)$ on $\bbC_+$ is a discrete $m$-function if and only if
%\begin{equation*}
%m(z)=\int_\bbR \frac{d\mu(x)}{x-z}
%\end{equation*}
%for some probability measure $\mu$ on $\bbR$ with bounded support.
%\end{lemma}

%%%%%%%%%%%%%%%%%%%%%%%%%%%%%%%%%%%%%%%%%%%%%%%%%%%%%%%%%%%%%%%%%%%%%%%%%

Let $\mathfrak{m}$ be a Herglotz function. Assume that the corresponding measure $\mu$ has $\esssup\mu=\fre$, a finite gap set. Denote the associated Riemann surface by $\calS$. Then $\mathfrak{m}$ is meromorphic on $(\bbC\cup\{\infty\})\setminus\fre$, which we identify with $\calS_+$. We are interested in conditions under which it has a continuation through the bands of $\fre$ to some region of $\calS_-$. The lemma below clarifies when this happens. The scalar result is due to Greenstein \cite{Greenstein}, while the matrix-valued can be found in \cite{Ges}.
\begin{lemma}\label{p4_prop}
Let $\mathfrak{m}$ be a matrix-valued Herglotz function with representation \eqref{p4_herg}. Then $\mathfrak{m}$ can be analytically continued from $\calS_+ \cap \pi^{-1}(\bbC_+)$ through an interval $I\subset\bbR$ if and only if the associated measure $\mu$ is purely absolutely continuous on $I$, and the density $f(x)=\frac{d\mu}{dx}$ is real-analytic on $I$. In this case, the analytic continuation of $\mathfrak{m}$ into some domain $\mathcal{D}_-$ of $\calS_-\cap \pi^{-1}(\bbC_-)$ is given by
\begin{equation*}
\mathfrak{m}(z_-)=\mathfrak{m}(\bar{z}_+)^*+2\pi i f(z), \quad z\in\pi(\mathcal{D}_-),
\end{equation*}
where $f(z)$ is the complex-analytic continuation of $f$ to some $\pi(\mathcal{D}_-)$.
\end{lemma}

Thus one can view any result on the continuation of $\mathfrak{m}$ as the corresponding result on the continuation of the absolutely continuous part $f$ of $\mu$.

%Moreover, now assume that $\mathfrak{m}$ has some continuation to some open neighborhood $\mathcal{D}$ of $I$ in $\calS_-$. This means that the two extensions into domains $\mathcal{D}\cap \calS_-\cap \pi^{-1}(\bbC_-)$ and $\mathcal{D}\cap \calS_-\cap \pi^{-1}(\bbC_+)$ have to agree on $\mathcal{D}\cap \calS_-\cap \pi^{-1}(\bbR)$. Note that since $\lim_{\veps\to0} \imag \mathfrak{m}(x-i \veps)=-\pi f(x)$, we have $\mathfrak{m}(z_-)=\mathfrak{m}(\bar{z}_+)^*-2\pi i f(\pi(z))$ for $z_-\in\mathcal{D}\cap \bbC_+$.  This means that the continuation of $f$ to $\pi(\mathcal{D}\setminus\bbR)$ has $f(z+i0)=-f(z-i0)$ for any $z\in \pi((\mathcal{D}\cap\bbR)\setminus I)$.
%Therefore $\mathfrak{m}$ has a continuation to some $\mathcal{D}\subset\calS_-$ if and only if $f$ can be continued to $\pi^{-1}(\mathcal{D})$ with $f(z_+)=-f(z_-)$ (in particular $f$ has to be zero or have a pole at the edge). Also note that this continuation satisfies $f^\sharp=f$ since $f$ is real on $\pi^{-1}(\fre)$.
\end{section}

%%%%%%%%%%%%%%%%%%%%%%%%%%%%%%%%%%%%%%%%%%%%%%%%%%%%%%%%%%%%%%%%%%%%%%%%%%%55

\end{appendix}

\bigskip

{\textbf{Acknowledgement}} The results were done during the author's stay at the California Institute of Technology. The author thanks Prof Barry Simon for the introduction to the problem and numerous useful discussions.

\def\cprime{$'$}


\begin{thebibliography}{}

\bibitem{AN}
Alexander~I. Aptekarev and Evgenii~M. Nikishin.
\newblock The scattering problem for a discrete {S}turm-{L}iouville operator.
\newblock {\em Mat. Sb. (N.S.)}, 121(163)(3):327--358, 1983.

\bibitem{Ball}
Joseph~A. Ball, Israel Gohberg, and Leiba Rodman.
\newblock {\em Interpolation of rational matrix functions}, volume~45 of {\em
  Operator Theory: Advances and Applications}.
\newblock Birkh\"auser Verlag, Basel, 1990.

\bibitem{Baumgartel}
Hellmut Baumg{\"a}rtel.
\newblock {\em Analytic perturbation theory for matrices and operators},
  volume~15 of {\em Operator Theory: Advances and Applications}.
\newblock Birkh\"auser Verlag, Basel, 1985.

\bibitem{DKS}
David Damanik, Rowan Killip, and Barry Simon.
\newblock Perturbations of orthogonal polynomials with periodic recursion
  coefficients.
\newblock {\em Ann. of Math. (2)}, 171(3):1931--2010, 2010.

\bibitem{DPS}
David Damanik, Alexander Pushnitski, and Barry Simon.
\newblock The analytic theory of matrix orthogonal polynomials.
\newblock {\em Surv. Approx. Theory}, 4:1--85, 2008.

\bibitem{DS2}
David Damanik and Barry Simon.
\newblock Jost functions and {J}ost solutions for {J}acobi matrices. {II}.
  {D}ecay and analyticity.
\newblock {\em Int. Math. Res. Not.}, Art. ID 19396, 32 pages, 2006.


\bibitem{Geronimo_matrix}
Jeffrey~S. Geronimo.
\newblock Scattering theory and matrix orthogonal polynomials on the real line
\newblock {\em Circuits Systems Signal Process.}, 1(3--4):471--495, 1982.

\bibitem{Geronimo}
Jeffrey~S. Geronimo.
\newblock Scattering theory, orthogonal polynomials, and {$q$}-series.
\newblock {\em SIAM J. Math. Anal.}, 25(2):392--419, 1994.


\bibitem{Ges}
Fritz Gesztesy and Eduard Tsekanovskii.
\newblock On matrix-valued {H}erglotz functions.
\newblock {\em Math. Nachr.}, 218:61--138, 2000.

\bibitem{Greenstein}
David~S. Greenstein.
\newblock On the analytic continuation of functions which map the upper half
  plane into itself.
\newblock {\em J. Math. Anal. Appl.}, 1:355--362, 1960.

\bibitem{IK3}
Alexei Iantchenko and Evgeny Korotyaev.
\newblock Periodic {J}acobi operator with finitely supported perturbation on
  the half-lattice.
\newblock {\em Inverse Problems}, 27(11):115003, 26, 2011.

\bibitem{IK1}
Alexei Iantchenko and Evgeny Korotyaev.
\newblock Periodic {J}acobi operator with finitely supported perturbations: the
  inverse resonance problem.
\newblock {\em J. Differential Equations}, 252(3):2823--2844, 2012.

\bibitem{IK2}
Alexei Iantchenko and Evgeny Korotyaev.
\newblock Resonances for periodic {J}acobi operators with finitely supported
  perturbations.
\newblock {\em J. Math. Anal. Appl.}, 388(2):1239--1253, 2012.

\bibitem{Kato}
Tosio Kato.
\newblock {\em Perturbation theory for linear operators}.
\newblock Classics in Mathematics. Springer-Verlag, Berlin, 1995.

\bibitem{K_thesis}
Rostyslav Kozhan.
\newblock {\em Asymptotics for {O}rthogonal {P}olynomials, {E}xponentially
  {S}mall {P}erturbations and {M}eromorphic {C}ontinuations of {H}erglotz
  {F}unctions}.
\newblock ProQuest LLC, Ann Arbor, MI, 2010.
\newblock Thesis (Ph.D.)--California Institute of Technology.

\bibitem{K_equiv}
Rostyslav Kozhan.
\newblock Equivalence classes of block {J}acobi matrices.
\newblock {\em Proc. Amer. Math. Soc.}, 139(3):799--805, 2011.

\bibitem{K_jost}
Rostyslav Kozhan.
\newblock Jost asymptotics for matrix orthogonal polynomials on the real line.
\newblock {\em Constr. Approx.}, 36(2):267--309, 2012.

\bibitem{S_merom_jost}
Barry Simon.
\newblock Meromorphic {J}ost functions and asymptotic expansions for {J}acobi
  parameters.
\newblock {\em Funct. Anal. Appl.}, 41(2):143--153, 2007.

\bibitem{Rice}
Barry Simon.
\newblock {\em Szeg{\H o}'s theorem and its descendants: spectral theory for
  $L{^{2}}$ perturbations of orthogonal polynomials}.
\newblock M. B. Porter Lectures. Princeton University Press, Princeton, NJ,
  2011.

\end{thebibliography}
\end{document}